\title{ Expressing an observer in preferred  coordinates by transforming an injective immersion into a surjective diffeomorphism }
\author{Pauline Bernard, Vincent Andrieu and Laurent Praly
\thanks{% 
P.~Bernard \textsf{Pauline.Bernard@mines-paristech.fr}
and L.~Praly \textsf{Laurent.Praly@mines-paristech.fr} are with
MINES ParisTech,
PSL Research University,
CAS - %
France,\hfill \break\null\hskip 2\parindent
V.~Andrieu \textsf{vincent.andrieu@gmail.com} is with LAGEP, CNRS, CPE, Universit\'e Lyon 1, France}
}
\documentclass[final]{siamltex704}

\usepackage{amsfonts,latexsym,graphicx,amsmath,amssymb,color}
\usepackage{hyperref}
\usepackage[dvipsnames]{xcolor}
\usepackage{tikz}
\usetikzlibrary{shapes}

\definecolor{journalcolor}{rgb}{1,0,0}
\definecolor{longuecolor}{rgb}{0,0,1}

\long\def\journalOK#1{\begingroup\color{journalcolor}#1}
\long\def\journalNO#1{\begingroup}
\long\def\longueOK#1{\begingroup\color{longuecolor}#1}
\long\def\longueNO#1{\begingroup}

\def\stopjournal{\endgroup}
\def\stoplongue{\endgroup}

\def\journal{%
\def\entete{\color{journalcolor} Journal version, compiled on \today}
\let\startjournal=\journalOK
\let\startlongue=\longueNO
}
\def\longue{%
\def\entete{\color{longuecolor} Long version, compiled on \today}
\let\startjournal=\journalNO
\let\startlongue=\longueOK
}
\def\deux{%
\def\entete{\color{journalcolor} Double version, compiled on \today}
\let\startjournal=\journalOK
\let\startlongue=\longueOK
}

\def\NN{{\mathbb N}}    % set of natural number
\def\RR{{\mathbb R}}    % field of real number
\def\SS{{\mathbb S}}    % sphere

\def\Ouvt{\mathcal{O}}
\def\Ouvs{\tilde{\mathcal{O}}} % Set of interest  for system state
\def\Ouvxw{\mathcal{O}_a}

\def\OuvEx{\mathcal A}  % Open set of the state space manifold
\def\OuvDef{\mathcal{O}_a} % Open set of the extended state
\def\CL{\texttt{cl}}    % Closure of a set
\def\BR{\mathcal{B}}    % Open ball
\def\comp{^\mathsf{c}}  % Complement
\def\Eeps{E_\varepsilon} % 
\def\ox{{\hat{\xi }}}   % State of the observer
\def\oX{{\hat{\Xi }}}   % State trajectory of the observer system
\def\ohi{{\tau^*}}   % Mapping from the state space toward the contracting coordinae{

\def\ohie{{\tau _e^*}}      % Extension of the function from the state space toward the contracting coordinate
\def\ohia{{ \tau _a^* }}      % Extension of the function from the state space toward the contracting coordinate
\def\ohede#1{{\tau _{e#1}}} 

\def\oh{\tau }          % Mapping from the contracting coordinates toward the state space (observer output function)
\def\ohe{\tau _e}          % Mapping from the contracting coordinates toward the state space (observer output function)
\def\of{\varphi}      % Observer dynamic

\def\lambdaa{\lambda}
\def\lambdae{\lambda_e}

\def\vartau{{\scriptstyle\mathfrak{T}}}

\def\temps{\nu}

\def\Id{Id} %identité

\def\varchi{z}

\def\E{\mathcal{E}}

\DeclareMathAlphabet{\EuScript}{U}{eus}{m}{n}
\SetMathAlphabet{\EuScript}{bold}{U}{eus}{b}{n}

\def\Lyap{\EuScript{V}} 

\def\Supp{\texttt{Supp}}

\def\setphitau{\mbox{$\varphi\mkern -5 mu$}\raise - 0.05em \hbox{\small$\mathcal{T}$}\mkern -4 mu}

\def\R{\Re}

\def\sat{\texttt{sat}}

\catcode`\@=11
\def\downparenfill{$\m@th\braceld\leaders\vrule\hfill\bracerd$}
\def\overparen#1{\mathop{\vbox{\ialign{##\crcr\crcr
\noalign{\kern0.3ex}
\downparenfill\crcr\noalign{\kern0.5ex\nointerlineskip}
$\hfil\displaystyle{#1}\hfil$\crcr}}}\limits}
\catcode`\@=12

\newtheorem{remark}{Remark}
\newtheorem{problem}{Problem}
\newtheorem{example}{Example}

\def\condC{\mathbb{B}} % conditions for a set to be diffeomorphic to \RR^n

\def\deps{ 2\varepsilon}
\def\Edeps{ E_{2\varepsilon}}
\def\Seps{\Sigma _\varepsilon }
\def\Sdeps{\Sigma _{2\varepsilon }}
\def\Sepscomp{\left(\Sigma _\varepsilon \right)^\mathsf{c}}
\def\Sdepscomp{\left(\Sigma _{2\varepsilon }\right)^\mathsf{c}}

\def\sde{\mathfrak{s}}

\def\jacobcomplet{}

\def\tde{\mathfrak{t}}

\def\hypo{\mathbb{A}}

\def\colorforrevision{\color{black}}
\definecolor{modifPBcolor}{rgb}{0.6,0.2,0.4}
\definecolor{modifcolor}{rgb}{0.2,0.6,0.4}

\def\startmodif{\begingroup\color{black}} % For Laurent

\def\stopmodif{\endgroup}

\longue
\begin{document}
\maketitle
\thispagestyle{myheadings}

\markboth{\entete}{\entete}

\begin{abstract}
When designing observers for nonlinear systems, the dynamics of the given system and of the designed observer are usually not expressed in the same coordinates or even have states evolving in different spaces. In general, the function, denoted $\oh$ (or its inverse, denoted $\ohi$) giving one state in terms of the other is not explicitly known and this creates implementation issues.

We propose to round this problem by expressing the observer dynamics in the the same coordinates as the given system. But this may impose to add extra coordinates, problem that we call \textit{augmentation}. This may also impose to modify the domain or the range of the ``augmented'' $\oh$ or $\ohi$, problem that we call \textit{extension}.

We show that the augmentation problem can be solved partly by a continuous completion of a free family of vectors and that the extension problem can be solved by a function extension making the image of the extended function the whole space. We also show how augmentation and extension can be done without modifying the observer dynamics and therefore with maintaining convergence.

Several examples illustrate our results.
\end{abstract}

\section{Introduction}

\subsection{Context}
In many applications, estimating the state of a dynamical system is
crucial either to build a controller or simply to 
obtain real time information on the system. 
Satisfactory solutions are
known for  systems the dynamics of which are linear in the  preferred  coordinates. But when they are nonlinear,
we are aware of only two ``general
purpose'' observer design methodologies guaranteeing ``non local''
convergence under merely some basic observability
properties: the high gain observers (\cite{KhaSab, Tor, GauHamOth,
GauKup,KhaPra,BesTic}, \ldots) and the nonlinear Luenberger observers
(\cite{Shos,KazKra,AndPra}). For both, the observer state is 
living in a space different from the system state one
and the system state estimate is obtained typically by solving on-line a 
nonlinear equation.

As an illustration, consider 
an harmonic oscillator with unknown frequency with dynamics
\\\null \hfill $\displaystyle 
\dot x_1 = x_2
\  ,\ 
\dot x_2 = -x_1 x_3
\  ,\
\dot x_3 = 0
\  ,\
y=x_1
$\refstepcounter{equation}\label{systOscil}\hfill$(\theequation)$\\[0.7em]
with state $x=(x_1,x_2,x_3)$ in 
%%% BEFORE
%$\RR^2\times\RR_{> 0}$
\startmodif $\left(\RR^2\setminus\{(0,0)\}\right)\times\RR_{> 0}$ \stopmodif and measurement $y$. 
\startmodif
We are interested 
in estimating the state $x$ from the only knowledge of 
$y$ and  the fact that $x$ evolves in some known set
$\OuvEx$. 
\stopmodif
% BEFORE
%We are interested 
%in estimating as $\hat x$ the state $x$ from the only knowledge of 
%$y$ and maybe the fact that $x$ evolves in some known set
%$\OuvEx$. 
By following in a very 
orthodox way (see \cite{AndEytPra} for details) the high gain observer design
we get a ``raw'' observer with dynamics
\begin{equation}
\label{LP3}
\dot \ox \;=\; 
\of (\ox,\hat x,y)\;=\; \left(\begin{array}{@{}cccc@{}}
0 & 1 & 0 &0
\\
0&0&1 & 0
\\
0&0&0&1
\\
0&0&0&0
\end{array}\right)\ox
+
\left(\begin{array}{@{}c@{}}
0 \\ 0 \\ 0 \\ \sat(\hat x_1\hat x_3^2)
\end{array}\right)
+
\left(\begin{array}{@{}c@{}}
\ell k_1 \\ \ell^2 k_2 \\ \ell^3 k_3 \\ \ell ^4 k_4
\end{array}\right)[y-\ox_1]
\  ,
\end{equation}
with state $\ox$ in $\RR^4$, where $\sat$ is a saturation function
(see (\ref{LP13})), and from which
 the system state 
estimate
$\hat x$ is 
%%%%%%%%%%%%% BEFORE %%%%%%%%%%%%%%%%%%%%%%%%%%%%%%%%%%%%%%%%%%%%%%%%%%%%%%%%%%%%%%%%%%%%%%%%%%%%%
%the system state 
%estimate
%$\hat x$ is obtained by solving in $\hat x$
%\begin{equation}
%\label{eq_ExplHGImm}
%\ox=\left(\ox_1,\ox_2,\ox_3,\ox_4\right)
%\;=\; \ohi(\hat x)=\left(
%\hat x_1 , \hat x_2 ,  -\hat x_1\hat x_3 , -\hat x_2\hat x_3\right)
%\  .
%\end{equation}
%This is a system of $4$ equations in $3$ unknowns which in general 
%has no exact solution. To get an approximate solution, 
%it is commonly proposed to solve (on-line)
%$$
%\hat x = \oh(\ox)=  \textnormal{Arg}\!\min_{\hat x}
%\left|\ox-\ohi(\hat x)\right|^2
%\ .
%$$
%This $\oh(\ox)$ is uniquely defined on the image
%of $\left(\RR^2\setminus\{(0,0)\}\right)\times \RR_{>0}$ by $\ohi$
% as
%$
%\oh(\ohi(x))\;=\; x
%$,
%but not necessarily outside and solving this minimization problem can be costly in practice. Instead, in \cite{RapMal}, the authors propose to build a global Lipschitzian extension of the function $\oh$. 
%%%%%%%%%%%%%%%%%%%%%%%%%%%%%%%%%%%%%%%%%%%%%%%%%%%%%%%%%%%%%%%%%%%%%%%%%%%%%
\startmodif 
given as $\hat x = \oh(\ox)$ %where 
where $\tau$ is any continuous function which satisfies
\begin{equation}
\label{eq_ExplHGImm}
\tau(x_1,x_2,-x_1x_3,-x_2x_3)=(x_1,x_2,x_3) \qquad \forall x=(x_1,x_2,x_3)\in \mathcal{A} \ .
%\left\{\begin{array}{l}
%\oh(\ohi(x))= x, 
%\\\ohi(x)=\left(
%x_1 , x_2 ,  -x_1 x_3 , - x_2 x_3\right)
%\end{array}\right.\ , x\in\OuvEx\ .
\end{equation}
The construction of the mapping $\oh$ relies on the inversion to the mapping 
$\ohi(x)=\left(x_1 , x_2 ,  -x_1 x_3 , - x_2 x_3\right)$
which in general has  no explicit solution and is not uniquely defined outside 
of $\ohi(\OuvEx)$. 
The commonly used implicit solution is given as the solution to an
%we may 
%go with solving
 optimization problem which may be 
$$
\hat x = \oh(\ox)=  \textnormal{Arg}\!\min_{\hat x}
\left|\ox-\ohi(\hat x)\right|^2
\ .
$$
Note however that some other forms are possible.
For instance in \cite{RapMal}, the authors propose to build another implicit solution based on an optimization procedure which yields a global Lipschitz function $\oh$.
The drawback of all these optimization based approaches being that they may be costly to solve in practice.
\stopmodif
Another path is to rely on the Rank theorem,
as in \cite{LevMar}
and  take advantage of
the 
local existence of diffeomorphism $\phi_x$ and $\phi_\xi$ such that
$$
\phi_\xi\circ\tau^*\circ \phi_x =(x,0, \hdots, 0) \ .
$$ 
In this case, one can pick
$
\hat x =\phi_x^{-1}( \pi (\phi_\xi(\ox))
$
where $\pi$ is the projection on the set of the first $n$ components. 
 In 
our example,
%we can select
%$\phi_x$ as the identity and
%along these lines, 
$\phi_x$ could be the identity and
$$
\phi_\xi(\xi)=
\left(\xi_1 ,
     \xi_2,
         -\frac{\xi_1\xi_3 + \xi_2 \xi_4}{\xi_1^2 + \xi_2^2}, 
                (\xi_1\xi_4 - \xi_2 \xi_3)\right)
\ .
$$
But, besides the local nature of this technique, finding expressions for $\phi_x^{-1}$ and $\phi_\xi$ may be a very 
difficult task in practice   (see \cite{LafBusGau} for instance). 
And unfortunately $\hat x$ %this expression 
is needed to evaluate
the term $\sat(\hat x_1\hat x_3^2)$ in (\ref{LP3}) since the observer dynamics depend on $\tau$.

Instead of a high gain observer design as above, we may use a Luenberger 
non linear observer design (see \cite{Shos,KazKra,AndPra}). It leads to~:
\begin{equation}
\label{8}
\dot \ox \;=\; \of(\ox,y)\;=\; A\,  \ox \;+\; B\,  y
\end{equation}
with $\ox$ in $\RR^4$, $A$ a Hurwitz matrix and $(A,B)$ 
a controllable pair. 
%%%%%% BEFORE %%%%%%%%
%In this case the function $\ohi$ used to obtain $\hat x$ is,
%with again $4$ equations in $3$ unknowns,
%\begin{equation}
%\label{2}
%\ohi(x)\;=\; -(A^2 + x_3I)^{-1}[ABx_1+Bx_2]
%\  .
%\end{equation}
\startmodif
The state estimate $\hat x$ is again given as $\hat x = \oh(\ox)$ where 
$\oh$ is any continuous function which satisfies $\oh(\ohi(x))=x$ for all $x$ in $\OuvEx$ where this time,
\begin{equation}
\label{2}
\ohi(x)\;=\; -(A^2 + x_3I)^{-1}[ABx_1+Bx_2]
\  .
\end{equation}
\stopmodif
A difference with the high gain observer is that 
$\hat x$ is not involved in (\ref{8}), i.e. the observer dynamics do not depend on $\tau $.

\startmodif
In the following, instead of constructing the (implicit) function  $\tau $ by a %an explicit
 minimization of a criterion introduced as a design tool, we explicitly construct a diffeomorphism $\tau _e$ allowing us to express the dynamics of the observer in the $x$-coordinates\footnote{% 
We will refer to the $x$-coordinates as the ``preferred coordinates'' or ``given coordinates'' because they are chosen by the user to describe the model dynamics. 
}.
\stopmodif
%%%%%%%%%%%%%%%   BEFORE  %%%%%%%%%%%%
%In the following, we
%propose a methodology to write the dynamics of the given observer \eqref{LP3} directly in the
%$x$-coordinates\footnote{%
%We will also refer to the $x$-coordinates as the ``  preferred  coordinates'' or ``  given  coordinates'' because they are chosen by the user to describe the model dynamics.
%%
%} in order to eliminate the minimization step.  
%%%%%%%%%%%%%%%%%%%%%%%%%%%%%%%%%%%%%%%%%%%%%%%%%%%%%%%%%%%%%%%
This has been suggested by several researchers
\cite{DezaBusGauRak,MagPas,AstPra} in the case where the observer state $\hat{\xi}$
and the state estimate 
$\hat x$ are related by a diffeomorphism. We
remove this restriction and complete the
preliminary results presented in \cite{AndEytPra}.  

In the example  above,
pulling the observer dynamics in the $\xi$-coordinates back in the $x$-coordinates is
seemingly impossible
since $x$ has dimension $3$ 
whereas $\ox$ has dimension $4$. 
%We overcome this difficulty by adding 
To overcome this difficulty, one could think of using again some kind of projection/restriction. Our proposition is actually of a completely different kind. Instead of considering $\ox$ as the estimation of the image by an immersion $\ohi$ of the state $x$, we see it as the estimation of the image by a diffeomorphism $\ohie$ of an augmented state $(x,w)$.
Fortunately with such a diffeomorphism $\ohie$, we can use all what has been proposed for expressing the observer dynamics in the preferred
coordinates in that case. So with this augmentation of $x$ into $(x,w)$, the design of the commonly used
projection/restriction is replaced by the construction of the diffeomorphism $\ohie$.
We show in Section \ref{jac_comp} that
$\ohie$ can be obtained by
``augmenting'' the function 
$x\mapsto \ohi(x)$ given in (\ref{eq_ExplHGImm})
or (\ref{2}). For this, it turns out that it is sufficient to
complement
a full column rank Jacobian into an invertible matrix.

The drawback of this approach however is that, 
because it is  linked to particular coordinate systems, 
the obtained diffeomorphism may not be defined everywhere. Also, its image could be only a subset of the observer accessibility set (for $\ox$), namely  the trajectories of $\ox$ may leave  the image of the diffeomorphism or equivalently 
 the trajectories of $(\hat x,\hat w)$ may leave
the domain of definition of the diffeomorphism. We show in Section \ref{diffeo_ext} how this 
new problem can be overcome via an extension of the image of the  
diffeomorphism. The key point here is that
the given observer dynamics (\ref{LP3}) remain unchanged. Hence we deal with constraints on the observer 
state without any kind of projection/restriction as commonly proposed (see
\cite{MagPas,AstPra} for example). A benefit of this is that, to preserve the convergence property,
we do not require extra assumptions such as convexity .

To illustrate
our results, we continue the example of the harmonic oscillator with unknown frequency
and 
add one based on
the bioreactor presented in \cite{GauHamOth}. 
We use a high-gain observer as starting point.
But, as shown in \cite{BerPraAnd}, the same tools can be used with a nonlinear Luenberger observer.
%
%These are done with using
%a high-gain observer as raw observer. 
%Employing the same tools, a nonlinear Luenberger observer is considered in \cite{BerPraAnd}.

Our contribution relies on, or is inspired by ideas of some known analysis results such as
continuously completing an independent set of vectors to a basis
\cite{Waz,Eck}, diffeotopies \cite{Hir} or $h$-cobordism \cite{Milnor}. We rephrase part of them when 
it is constructive and therefore useful for observer design. Similarly, the constructive part of 
our proofs are in the main body of our text,
 those which are not constructive and never used/commented in remarks or examples are in appendix or omitted 
 to save space. 
 \startjournal{A more complete version with all the proofs is in \cite{BerPraAndSIAMHal}.}
\stopjournal
\startlongue{
This is the long version of a paper which has been submitted for publication in SIAM Journal of Control and optimization.
The parts of the paper which are in blue are those which are modified with respect to the journal version.
}
\\
\stoplongue

\subsection{Problem statement}

We consider
the given system with
dynamics~:
\begin{equation}\label{syst}
\dot x = f(x)\quad, \qquad y=h(x)\ ,
\end{equation}
with $x$ in
$\RR^n$ and $y$ in $\RR^q$. Its
solution at time $t$, with initial condition $x_0$ at time $0$ is denoted $X(x_0,t)$ and the corresponding output $y_{x_0}(t)$.
The observation problem is to construct a dynamical system with input $y$ and output $\hat x$, supposed to be
an estimate of the system state $x$
as long as the latter
is in a specific set of interest denoted $\OuvEx\subseteq\RR^n$. 
\startmodif As starting point,  \stopmodif we
assume this problem is (formally) already solved but with maybe
some implementation issues
such as finding an expression of $\oh$.
More precisely,

\itshape \underline{Assumption $\hypo$} (Converging observer)~: There exist an open subset $\Ouvt $ of
$\RR^n$ containing $\OuvEx$, a $C^1$ injective immersion $\ohi:\Ouvt\to \RR^m$, and a set\footnote{%
% BEGIN FOOTNOTE
The symbol $\setphitau$ is pronounced \textit{phitau}.
% END FOONOTE
} $\setphitau$ of pairs $(\of,\oh)$ of
locally Lipschitz functions such that we have
\\[0.5em]\null \hfill $\displaystyle 
\oh (\ohi (x))\;=\; x
\qquad \forall x\in \OuvEx
$\hfill \null 
\refstepcounter{equation}\label{ObsDynLeftInverse}$(\theequation)$
\\[0.7em]
and, for any solution $X(x_0,t)$ of (\ref{syst}) which is defined
and remains in $\OuvEx$ for $t$ in $[0,+\infty )$, the solution $(X(x_0,t),\oX(\ox_0,t;y_{x_0}))$ of the cascade
system~:
\begin{equation}\label{eq_sysObs}
	\dot x \;=\; f(x)\quad ,\qquad
y\;=\; h(x)\quad ,\qquad 
	\dot \ox = \of (\ox,\hat x ,y)
\quad ,\qquad 
\hat x=\oh(\ox)
	\  ,
\end{equation}
with initial condition $(x_0,\ox_0)$ in $\OuvEx\times\RR^m$ at time 
$0$, is also defined on $[0,+\infty )$ and  satisfies~:
\begin{equation}
\label{9}
\lim_{t\to +\infty }
\left|\oX(\ox_0,t;y_{x_0})-\ohi(X(x_0,t))\right|\;=\; 0
\  .
\end{equation}
\normalfont
\par\vspace{0.5em}

\begin{remark}~\normalfont
\label{rem1}
\setcounter{enumi}{0}
\begin{list}{}{%
\parskip 0pt plus 0pt minus 0pt%
\topsep 0.5ex plus 0pt minus 0pt%
\parsep 0pt plus 0pt minus 0pt%
\partopsep 0pt plus 0pt minus 0pt%
\itemsep 0.5ex plus 0pt minus 0pt
\settowidth{\labelwidth}{1.}%
\setlength{\labelsep}{0.5em}%
\setlength{\leftmargin}{\labelwidth}%
\addtolength{\leftmargin}{\labelsep}%
}
\refstepcounter{enumi}\label{point4}\item[\theenumi.]
The convergence property given by (\ref{9}) is in the observer state space only. Property 
(\ref{ObsDynLeftInverse}) is a necessary condition for this convergence to be transferred from the observer state space to the system state space.
\refstepcounter{enumi}\label{point3}\item[\theenumi.]
The need for pairing $\of$ and $\oh$ comes from the dependence on $\hat x=\oh(\ox)$ of $\of$ in 
(\ref{eq_sysObs}).
This may imply to change $\of$ whenever we change $\oh$.
In the high-gain approach, as in (\ref{LP3}), when $\OuvEx$ is bounded, thanks to the 
gain $\ell$ 
which can be chosen arbitrarily large, $\of $ can be 
paired with any locally  Lipschitz function $\oh$
provided its values are saturated 
whenever they are used as arguments of $\of $. On another hand,
if, as in  (\ref{8}), $\of$ does not depend on $\hat x$, then it can be paired with any $\tau$.
\end{list}
\end{remark}

\begin{example}\normalfont
\label{ex_osci_0}
For System
%harmonic oscillator with unknown frequency
(\ref{systOscil}),
for any solution with initial condition $x_1=x_2=0$, we 
have no information on $x_3$ from
the only knowledge of (\ref{systOscil}) and the function $t\mapsto 
y(t)=X_1(x,t)$. This explains the restriction of our attention to
the set 
\begin{equation}\label{eq_OuvsExpl}
\OuvEx\;=\; \left\{x\in\RR^3\,  :\: x_1^2+x_2^2 \in 
\left]\frac{1}{r},r\right[\: ,\;
x_3\in ]0,r[\right\}
\  ,
\end{equation}
where $r$ is some arbitrary strictly positive real number.
This set is invariant
by (\ref{systOscil}),
and  the 
function (\ref{eq_ExplHGImm})
being an injective immersion on 
$\left(\RR^2\setminus\{(0,0)\}\right)\times \RR_{>0}$,
the system is strongly differentially observable%
\footnote{%\par\vspace{1em}\noindent
% BEGIN FOOTNOTE
The system is said to be strongly differentially observable of order $m$
if the function $x\mapsto(h(x),L_fh(x),...,L_f^{m-1}h(x))$ is an injective immersion.
% END FOONOTE
} of order 4 on this set. 
Let $\Ouvt$ be any open subset such that
$\CL(\OuvEx)\subset \Ouvt \subseteq
\left(\RR^2\times\RR_{>0}\right)
\setminus\left(\{(0,0)\}\times\RR_{>0}\right)$,
with $\CL$ denoting the set closure.
Then,
 $\CL(\OuvEx)$ being a compact set, a set $\setphitau$ 
satisfying Assumption $\hypo$ is made of
pairs of a locally Lipschitz function $\oh$ satisfying
(see 
\cite{KhaPra} for example)
\begin{equation}\label{eq_phitau}
x=\oh(x_1,x_2,-x_1x_3,-x_2x_3)
\qquad \forall x\in \OuvEx
\end{equation}
and the
function $\of $ defined in
(\ref{LP3}) where
\begin{equation}
\label{LP13}
\sat(s)\;=\; \min\left\{r^3,\max\left\{s,-r^3\right\}\right\} 
\end{equation}
with the gain
$\ell$ in (\ref{LP3})
adapted to the properties of $\oh$.
\hfill$\triangle$\end{example}

Although the problem of observer design seems already solved under Assumption $\hypo$, it can be difficult to find a left-inverse $\oh$ of $\ohi$. In the following, we consider that the function $\ohi$ and the set 
$\setphitau$ are given and we aim at avoiding the left-inversion of $\ohi$ by expressing the observer for $x$ in 
the, maybe augmented, $x$-coordinates. 
More precisely we aim at solving the following problem.

\itshape \underline{Our problem} (Observer in the $x$-coordinates) ~:
 Assume that Assumption $\hypo$ is satisfied, we wish to find an open set  $\Ouvxw\subseteq\RR^{m}$ and two mappings $k$ and $\ell$ such that the system defined in $\RR^{m}$
\begin{equation}\label{ObsOriCoord}
\dot {\hat x} =  k(\hat x,\hat w,y)\ ,\ \dot {\hat w} = \ell(\hat x,\hat w,y)\ ,
\end{equation}
defines an observer in $\OuvEx$.
In other words, for any initial condition $x_0$ in $\OuvEx$ such that the solution $X(x_0,t)$ of (\ref{syst}) is defined and remains in $\OuvEx$ for $t$ in $[0,+\infty )$, the solution $(X(x_0,t),\hat X(\hat x_0,\hat w_0,t;y_{x_0}),\hat W(\hat x_0,\hat w_0,t;y_{x_0}))$, with initial condition $(\hat x_0,\hat w_0)$ in $\Ouvxw$, of the cascade of system (\ref{syst}) with the observer (\ref{ObsOriCoord})
 is also defined on $[0,+\infty )$ and satisfies~: 
 \begin{equation} \label{ConvOrCoord} \lim_{t\to +\infty }  \left|X(x_0,t)-\hat X(\hat x_0,\hat w_0,t;y_{x_0})\right| = 0 \ . 
 \end{equation} 
\normalfont
\par\vspace{0.5em}
 
\subsection{A sufficient condition allowing us to express the 
observer in  the  given $x$-coordinates}
For the simpler 
case where the raw observer state $\ox$ has the same dimension as the 
system state $x$, i.e. $m=n$,
$\ohi$, in Assumption $\hypo$, 
is a diffeomorphism on $\Ouvt$ and we can express the observer in the  given $x$-coordinates as~:
\begin{equation}
\dot{\hat{x}} = \left(
\frac{\partial \ohi}{\partial x}(\hat x)
\right)^{-1}
\of (\ohi(\hat x),\hat{x},y) 
\label{ObsDynDiffeo}
\end{equation}
which requires a Jacobian inversion only. However, although,
by assumption, the system trajectories remain in $\Ouvt$ where the Jacobian
is invertible, we have no guarantee the
ones of the observer do.
Therefore, to obtain convergence and completeness of solutions, we must find means to ensure the estimate $\hat{x}$
does not leave the set $\Ouvt$, 
or equivalently that $\ohi(\hat{x})$ remains in the image set $\ohi(\Ouvt)$.
We address this point by modifying $\ohi$ ``marginally'' in order to get $\ohi(\Ouvt)=\RR^m$.

In the more complex situation where $m>n$, $\ohi$ is only an injective immersion.
In \cite{AndEytPra}, it is proposed to 
augment the given $x$-coordinates in $\RR^n$ with extra ones, say $w$, in $\RR^{m-n}$ and correspondingly to augment
 the given injective immersion 
$\ohi$ into a diffeomorphism $\ohie : \Ouvxw \rightarrow \RR^m$, 
where $\Ouvxw$ is an open  subset of $\RR^m$, considered as an augmentation of $\Ouvt$, i.e. its
Cartesian projection on $\RR^n$ is contained in $\Ouvt$ and contains $\CL(\OuvEx)$.

To help us find such an appropriate augmentation, we have the following 
sufficient condition.

\begin{proposition}\label{prop_CSobsWithoutInversion} 
Assume Assumption $\hypo$ holds and $\OuvEx$ is bounded. Assume also the existence of an
open subset $\Ouvxw$ of $\RR^m$ containing $\CL(\OuvEx\times\{0\})$ and of a  diffeomorphism $\ohie:\Ouvxw\to \RR^m$ satisfying 
\\[0.3em]\null \hfill $\displaystyle 
 \ohie(x,0)\;=\; \ohi(x)\qquad \forall x\in\OuvEx  
$\hfill \null 
\refstepcounter{equation}\label{LP15}$(\theequation)$
\\[0.3em]
 and 
\\[0.3em]\null \hfill $\displaystyle 
\ohie(\Ouvxw)=\RR^m
\  .
$\hfill \null 
\refstepcounter{equation}\label{n9}$(\theequation)$
\\[0.7em]
 and such that, with letting $\ohede{x}$ denote the $x$-component of the inverse of $\ohie$, there exists a 
 function $\of  $ such that the pair $(\of  ,\ohede{x})$ is in the set $\setphitau$ given by Assumption $\hypo$. 
 Under these conditions, for any initial condition $x_0$ in $\OuvEx$ such that the solution $X(x_0,t)$ of (\ref{syst}) is defined and remains in $\OuvEx$ for $t$ in $[0,+\infty )$, the solution $(X(x_0,t),\hat X(\hat x_0,\hat w_0,t;y_{x_0}),\hat W(\hat x_0,\hat w_0,t;y_{x_0}))$, with initial condition $(\hat x_0,\hat w_0)$ in $\Ouvxw$, of the cascade of system (\ref{syst}) with the observer~: 
 \begin{equation}\label{eq_RealObsDyn} 
 \dot {\overparen{ \left[ \begin{array}{c} \hat x\\ \hat w \end{array} \right] }} =\left( \frac{\partial \ohie}{\partial (\hat x,\hat w)}(\hat x,\hat w) \right)^{-1} \of (\ohie(\hat x,\hat w), \hat x,y) 
 \end{equation} 
 is also defined on $[0,+\infty )$ and satisfies~: 
 \begin{equation} \label{1} \lim_{t\to +\infty } \left|\hat W(\hat x_0,\hat w_0,t;y_{x_0})\right| + \left|X(x_0,t)-\hat X(\hat x_0,\hat w_0,t;y_{x_0})\right| = 0 \ . 
 \end{equation} 
 \end{proposition}
The key point in the observer (\ref{eq_RealObsDyn}) is that, instead of left-inverting the
function $\ohi$ via $\oh$ as in (\ref{ObsDynLeftInverse}), we invert only a
matrix. 
\begin{proof}
See Appendix \ref{app_CSobsWithoutInversion}.
\hfill\end{proof}

With Proposition \ref{prop_CSobsWithoutInversion}, we are left with finding
a diffeomorphism $\ohie$ satisfying the conditions listed in the statement~:
\begin{list}{}{%
\parskip 0pt plus 0pt minus 0pt%
\topsep 1ex plus 0pt minus 0pt%
\parsep 0pt plus 0pt minus 0pt%
\partopsep 0pt plus 0pt minus 0pt%
\itemsep 1ex plus 0pt minus 0pt
\settowidth{\labelwidth}{$\bullet$}%
\setlength{\labelsep}{0.5em}%
\setlength{\leftmargin}{\labelwidth}%
\addtolength{\leftmargin}{\labelsep}%
}
\item[$\bullet$]
Equation  (\ref{LP15}) is about the fact that $\ohie$ is an 
augmentation, with adding coordinates,
of the given injective immersion $\ohi$.
It motivates the following problem.

\begin{problem}[Immersion augmentation into a diffeomorphism]
\label{defP1}
Given a set $\OuvEx $, an open subset $\Ouvt$ of $\RR^n$ containing $\CL(\OuvEx)$, and  an injective immersion
$\ohi:\Ouvt\rightarrow \ohi(\Ouvt)\subset\RR^m$, the pair $(\ohia,\Ouvxw)$  is said to solve the problem of
immersion augmentation into a diffeomorphism if $\Ouvxw$ is an open subset of 
$\RR^m$ containing $\CL(\OuvEx\times \{0\} )$ and
$\ohia:\Ouvxw\rightarrow\ohia(\Ouvxw)\subset\RR^m$ is a diffeomorphism satisfying
$$
 \ohia(x,0)\;=\; \ohi(x)\qquad \forall x\in\OuvEx \ .
$$  
\end{problem}
We will present
in Section \ref{jac_comp} conditions under which Problem \ref{defP1}
 can
be solved via
complementing a full column rank Jacobian of $\ohi$ into an invertible matrix, i.e. via what we call
Jacobian complementation.

\item[$\bullet$]
The condition expressed in (\ref{n9}), is about the fact that $\ohie$ is 
surjective onto $\RR^m$. This motivates us to introduce the surjective diffeomorphism extension problem%
% \footnote{
% % BEGIN FOONOTE
% the surjectivity condition can be replaced by : 
% \textit{There exists a subset $\mathcal{S}$ of $\ohie(\Ouvxw)$ which is
% invariant by the dynamics 
% $\dot{\hat{\xi}}=\of (\hat{\xi},\oh(\hat{\xi}),y)$.
% }
% But dealing with this case is difficult when $\mathcal{S}$ 
% depends on $\oh$.
% % END FOONOTE
% }
\begin{problem}[Surjective diffeomorphism extension]
\label{defP2}
Given an open subset $\OuvDef$ of $\RR^m $, a compact subset $K$
of $\OuvDef$, and a diffeomorphism $\ohia$: $\OuvDef  \to  \RR^m$, the diffeomorphism
$\ohie : \OuvDef \rightarrow \RR^m$ is said to solve the surjective diffeomorphism extension problem if it satisfies
$$\ohie(\OuvDef)=\RR^m
\quad ,\qquad 
\ohie(\varchi)=\ohia(\varchi)\quad \forall \varchi \in K.
$$
\end{problem}
 This Problem \ref{defP2} will be addressed
 in Section \ref{diffeo_ext}.
\end{list}

When Assumption $\hypo$ holds and $\OuvEx$ is bounded, by successively solving Problem \ref{defP1} and
Problem \ref{defP2} with $\CL(\OuvEx\times \{0\} )\subset K\subset \Ouvxw$, we get a diffeomorphism
$\ohie$ guaranteed to satisfy all the conditions of Proposition \ref{prop_CSobsWithoutInversion} except maybe the fact that the pair
$(\of  ,\ohede{x})$ is in $\setphitau$.  How this last condition can be satisfied will be discussed in Section
\ref{sec1} mainly via a list of remarks.

Throughout Sections \ref{jac_comp}-\ref{diffeo_ext}, 
we will show how, step by step, we can express a high gain observer 
in the $x$-coordinates for the harmonic oscillator 
with unknown frequency. We will also show that our approach enables to ensure completeness of solutions of the observer presented in \cite{GauHamOth} for the bioreactor. The various difficulties we shall encounter on 
this road will be discussed in Section 
\ref{disc_P3}. In particular, we shall see how they 
can be overcome thanks to a
better choice  of $\ohi$ and of the pair $(\of  ,\tau)$ given by Assumption 
$\hypo$. 
%In Section \ref{ex_osci_luen}, we will see that the same tools can be applied for the construction of Luenberger observers. 

%%%%%%%%%%%%%%%%%%%%%%%%%%%%%%%%%%%%%%%%%%%%%%%%%%%%%%%%%%%%%%%%%%%%%%%%%%
%%%%%%%%%%%%%%%%%%%%%%%%%%%%%%%%%%%%%%%%%%%%%%%%%%%%%%%%%%%%%%%%%%%%%%%%%%

\section{About  Problem \ref{defP1} : Augmentation of an immersion into a diffeomorphism%%
\jacobcomplet% a Jacobian completion
}~
\label{jac_comp}
In \cite{AndEytPra}, we find the following sufficient condition for
the augmentation of an immersion into a diffeomorphism.

\begin{lemma}[\cite{AndEytPra}]
\label{lem_CS_P1}
Let $\OuvEx $ be a bounded set, $\Ouvt$ be an open subset of $\RR^n$ containing $\CL(\OuvEx)$, and  
$\ohi:\Ouvt\rightarrow \ohi(\Ouvt)\subset\RR^m$ be an injective immersion. If there exists
a bounded open set $\Ouvs$ satisfying
%\\[0.3em]\null \hfill 
$\displaystyle 
\CL(\OuvEx) \subset \Ouvs \subset \CL(\Ouvs) \subset \Ouvt
$ %\hfill \null \\[0.8em]
and a $C^1$ function $\gamma:\Ouvt\to \RR^{m\times(m-n)}$ 
the values of which are $m\times(m-n)$ matrices satisfying~:
\\[0.1em]\null \hfill $\displaystyle 
\det\left(
\frac{\partial \ohi}{\partial x}(x)
\quad
\gamma (x)
\right)\neq 0
\qquad \forall x \in \CL(\Ouvs)\  ,
$\refstepcounter{equation}\label{eq_MatComp}\hfill$(\theequation)$\\[0.7em]
then there exists
%a strictly positive real number $\varepsilon _\diamond$ 
%such that, for any 
a strictly positive real number $\varepsilon $
% in $(0,\varepsilon _\diamond]$,
such that the following pair\footnote{For a positive real number
$\varepsilon$
and $z_0$ in $\RR^p$, $\BR_\varepsilon(z_0)$ is the
open ball centered at $z_0$ and with radius
$\varepsilon$.
}  $(\ohia,\Ouvxw)$  solves Problem \ref{defP1}
\begin{equation}
\label{ohie}
\ohia(x,w)\;=\; \ohi(x) \;+\;
 \gamma (x)\,   w
\ ,\ \Ouvxw=\Ouvs\times\BR_{\varepsilon}(0)
\  .
\end{equation}
\end{lemma}
\indent
In other words, 
an injective immersion $\ohi$ can be augmented into a 
diffeomorphism $\ohia$
if we are able
to find $m-n$ columns $\gamma$ which are $C^1$ in $x$ and which
complement the full column rank Jacobian $\frac{\partial \ohi}{\partial x}(x)$ into an invertible matrix.

\begin{proof}
See Appendix \ref{app_CS_P1}.
\hfill\end{proof}

\begin{remark}
\label{rem_jacCompOrth}
\normalfont
Complementing a $m\times n$ full-rank matrix into an invertible one is equivalent to finding
$m-n$ independent vectors orthogonal to that matrix. Precisely the existence of $\gamma$
satisfying (\ref{eq_MatComp}) is equivalent to 
the existence of a $C^1$ function $\tilde \gamma:\CL(\Ouvs)\to \RR^{m\times(m-n)}$ the
values of which are full rank matrices
satisfying~:
\begin{equation}
\label{3}
\tilde \gamma (x)^\top
\frac{\partial \ohi}{\partial x}(x)
\;=\;  0
\qquad \forall x \in \CL(\Ouvs)\  .
\end{equation}
Indeed, $\tilde{\gamma}$ satisfying (\ref{3}) satisfies also (\ref{eq_MatComp}) since the following matrices 
are invertible
$$
\left(\begin{array}{c}
\frac{\partial \ohi}{\partial x}(x)^\top
\\
\tilde \gamma (x)^\top
\end{array}\right)
\left(
\frac{\partial \ohi}{\partial x}(x)
\quad
\tilde \gamma (x)
\right)\;=\; 
\left(\begin{array}{cc}
\frac{\partial \ohi}{\partial x}(x)^\top
\frac{\partial \ohi}{\partial x}(x)
& 0
\\
0 & \tilde \gamma (x)^\top
\tilde \gamma (x)
\end{array}\right)
\  .
$$
Conversely, given $\gamma $ satisfying
(\ref{eq_MatComp}), $\tilde \gamma $ defined by the identity below
satisfies (\ref{3}) and has full column rank
$$
\left(
\frac{\partial \ohi}{\partial x}(x)
\quad
\tilde \gamma (x)
\right)\;=\; 
\left(
\frac{\partial \ohi}{\partial x}(x)
\quad
 \gamma (x)
\right)
\left(\begin{array}{c@{\qquad }c}
I & -\left[\frac{\partial \ohi}{\partial x}(x)^\top
\frac{\partial \ohi}{\partial x}(x)\right]^{-1}\!\frac{\partial \ohi}{\partial x}(x)^\top
\gamma (x)
\\
0 & I
\end{array}\right) \ .
$$
\end{remark}

%In the following subsections, we give sufficient conditions and explicit methods to solve the problem of Jacobian complementation
%\jacobcomplet% a Jacobian completion
% either via (\ref{eq_MatComp}) or via (\ref{3}).

% In the following subsections, 
% we give sufficient conditions and explicit methods to solve Problem \ref{defP1}.
% According to Lemma \ref{lem_CS_P1} and Remark \ref{rem_jacCompOrth}, we need only to give sufficient conditions to complement a Jacobian
% %\jacobcomplet% a Jacobian completion
%  either via (\ref{eq_MatComp}) or via (\ref{3}).

\subsection{Submersion case}~
 \label{sec2}
\begin{proposition}[%
Completion when $\ohi(\CL(\Ouvs))$ is a level set of a submersion%
]
\label{Prop_Subm}
Let $\OuvEx $ be a bounded set, $\Ouvs$ be a bounded open set and $\Ouvt$ be an open  set
satisfying
$$
\CL(\OuvEx) \subset \Ouvs \subset \CL(\Ouvs) \subset \Ouvt
\  .
$$
Let also 
$\ohi:\Ouvt\rightarrow \ohi(\Ouvt)\subset\RR^m$ be an injective immersion.
Assume there exists a $C^2$ function $F$ : $\RR^m\rightarrow 
\RR^{m-n}$ which is a submersion at least
on a neighborhood of $\ohi(\Ouvs)$ satisfying:
\begin{equation}
\label{LP6}
F(\ohi(x))=0\qquad \forall x \in \Ouvs\  ,
\end{equation}
 then, with the $C^1$ function
$x\mapsto \gamma(x) = \frac{\partial F}{\partial \xi}^T(\ohi(x))$,
the matrix in (\ref{eq_MatComp}) is invertible for all $x$ in $\Ouvs$
and
the pair $(\ohia,\Ouvxw)$ defined  in (\ref{ohie})
 solves Problem \ref{defP1}.
\end{proposition}
\begin{proof}
For all $x$ in $\CL(\Ouvs)$,  $\frac{\partial \ohi}{\partial x}(x) $ 
is right invertible and 
we have
$
\frac{\partial F}{\partial \xi}(\ohi(x)) \frac{\partial 
\ohi}{\partial x}(x) =0
$.
Thus,
the rows
 of $\frac{\partial F}{\partial \xi}(\ohi(x))$ are orthogonal to the column vectors of $\frac{\partial \ohi}{\partial x}(x)$
and are independent since $F$ is a submersion.
The Jacobian of $\ohi$ can therefore be completed with $\frac{\partial F}{\partial \xi}^T(\ohi(x))$.
The proof is completed with Lemma \ref{lem_CS_P1}.
\hfill\end{proof}
 
\begin{remark}
\normalfont
Since $\frac{\partial \ohi}{\partial x}$ is of constant rank $n$ on 
$\Ouvt$, the existence of such a function $F$ is guaranteed at least locally 
by the constant rank Theorem. 
\end{remark}

\begin{example}[Continuation of Example \ref{ex_osci_0}]\normalfont
\label{ex_osci_1}
Elimination of the $\hat x_i$ in the $4$ equations given by the injective immersion $\ohi$ defined in (\ref{eq_ExplHGImm})
leads to the function 
$
F(\xi)=\xi_2 \xi_3 -\xi_1 \xi_4 
$
satisfying (\ref{LP6}). It follows that a candidate for 
complementing:
\\[0.7em]\null \hfill $\displaystyle 
\frac{\partial \ohi}{\partial x}(x) = 
\left(
\begin{array}{ccc}
1&0&0\\
0&1&0\\
-x_3 &0 & -x_1\\
0&-x_3&	-x_2
\end{array}
\right) 
$\refstepcounter{equation}\label{LP7}\hfill$(\theequation)$\\
is
%\\[0.3em]
\null \hfill $\displaystyle 
\gamma (x)\;=\; 
\frac{\partial F}{\partial \xi}(\ohi(x)) ^T= ( x_2x_3, -x_1 x_3, x_2, 
-x_1) ^T\ .
$\hfill \null \\[0.7em]
This vector is nothing but the column of the minors of 
the matrix (\ref{LP7}).
It gives as determinant 
$
(x_2x_3)^2+(x_1x_3)^2+x_2^2+x_1^2
$
which is never zero on $\Ouvt$. 

Then, it follows from
Lemma
 \ref{lem_CS_P1}, that, for any bounded open set $\Ouvs$ such that
$\OuvEx\subset\CL(\Ouvs)\subset\Ouvt$ the following function
is a diffeomorphism on $\Ouvs\times\BR_\epsilon(0)$ for $\varepsilon$ sufficiently small
\\[0.7em]\null \hfill $\displaystyle 
\ohia(x,w)=(x_1+x_2x_3w,x_2-x_1x_3w,-x_1x_3 +x_2w,-x_2x_3-x_1 w) \  .
$\hfill \null \\[1em]
\indent
With picking 
$\ohie=\ohia$, (\ref{eq_RealObsDyn}) gives us the following observer written in the
given
 $x$-coordinates
augmented
 with $w$ :
\\[0.7em]\null \hfill $\displaystyle 
\dot {\overparen{
\left(
\begin{array}{@{}c@{}}
\hat x_1\\\hat x_3\\
\hat x_2\\
\hat w
\end{array}
\right)
}}=\!\!
\left(\begin{array}{@{}cccc@{}}
1 & \hat x_3 \hat w & \hat x_2 \hat w & \hat x_2 \hat x_3
\\
-\hat x_3 \hat w & 1 & -\hat x_1 \hat w  & -\hat x_1 \hat x_3
\\
-\hat x_3 & \hat w & -\hat x_1 & \hat x_2
\\
-\hat w &-\hat x_3 & -\hat x_2 & -\hat x_1
\end{array} \right)^{\hskip -0.7em -1}
% \times
% $\hfill \null \\\null \hfill $\displaystyle 
\left[\!
\left(\begin{array}{@{}c@{}}
\hat x_2- \hat x_1 \hat x_3 \hat w
\\
-\hat x_1\hat x_3 +\hat x_2\hat w
\\
-\hat x_2\hat x_3-\hat x_1 \hat w
\\
\sat(\hat x_1\hat x_3^2)
\end{array}\right)
\!+\!
\left(\begin{array}{@{}c@{}}
\ell k_1 \\ \ell^2 k_2 \\ \ell^3 k_3 \\ \ell ^4 k_4
\end{array}\right)[y-\hat x_1]
\right]
$\hfill \null \\[0.7em]
Unfortunately the matrix to be inverted is non singular for
$(\hat x , \hat w )$ in $\Ouvs\times\BR_\varepsilon(0)$ only and we 
have no guarantee that the trajectories of this observer remain in this set. 
This shows that a further modification transforming $\ohia$ into $\ohie$
is needed to make sure that $\ohie^{-1}(\ox)$ belongs to this 
set whatever $\ox$ in $\RR^4$.
This is Problem \ref{defP2}.
\hfill$\triangle$\end{example}

The drawback of this Jacobian 
complementation method  is 
that it asks for the knowledge of the function $F$. It would be 
better to simply have a universal formula relating the entries of the 
columns to be added to those of $\frac{\partial \ohi}{\partial x}$.

\subsection{The $\tilde P[m,n]$ problem}
\label{jac_comp_priv}
Finding a universal formula for the Jacobian 
complementation problem
\jacobcomplet% a Jacobian completion
amounts to solving the following problem.
\begin{definition} $(\tilde P[m,n]$ problem)
For a pair of integers $(m,n)$ such that $0<n<m$, a  $C^1$ matrix function $\tilde{\gamma}:\RR^{m\times n} \to \RR^{m\times(m-n)}$ 
solves the $\tilde P[m,n]$ problem if for any $m  \times  n$ matrix
$\vartau=(\vartau_{ij})$
of rank $n$, 
the matrix $
\left(\begin{array}{@{}cc@{}} 
\vartau
& \tilde{\gamma}(
\vartau
)
\end{array}\right)
$
is invertible,
or equivalently,
the matrix $ \tilde{\gamma} (\vartau  )$ has rank $m-n$ and satisfies
$\   \tilde{\gamma} ( \vartau )^\top \vartau
\;=\; 0 \ .
$
\end{definition}

As a consequence of a theorem due to Eckmann \cite[\S 1.7 p. 126]{Eck} and Lemma \ref{lem_CS_P1}, we have
\begin{theorem}\label{theo_EckmannTilde}
The $\tilde P[m,n]$
problem is solvable by a $C^1$ function $ \tilde{\gamma}$ if and only if the pair $(m,n)$ is
one of the following $3$ pairs
\\[0.5em]\null \hfill $\displaystyle 
(>2,m-1)$ \hfill or\hfill $(4,1)$\hfill or\hfill $(8,1)\  .$\hfill 
\refstepcounter{equation}\label{LP11}\hfill$(\theequation)$
\\[0.5em]
% in the table
% \begin{equation}
% \label{LP11}
%  \renewcommand{\arraystretch}{1.5}
% \begin{array}{|r|c|c|c|c|}
% %\cline{2-5}
% \hline
% m = &\geq 2&4&8
% \\
% n= &m-1& 1&1
% \\
% %\cline{2-5}
% \hline
% \end{array}
% \end{equation}
Moreover, for each of 
these pairs and for
any bounded set $\OuvEx $, bounded open set $\Ouvs$ and open  set $\Ouvt$
satisfying
\\[0.7em]\null \hfill $\displaystyle 
\CL(\OuvEx) \subset \Ouvs \subset \CL(\Ouvs) \subset \Ouvt
\  ,
$\hfill \null \\[0.7em]
and any injective immersion
$\ohi:\Ouvt\rightarrow \ohi(\Ouvt)\subset\RR^m$,
the pair $(\ohia,\Ouvxw)$ defined in (\ref{ohie})
with $\gamma(x) = \tilde \gamma\left(\frac{\partial \ohia}{\partial x}(x)\right)$ solves Problem 
\ref{defP1}.\\
\end{theorem}
\startjournal{
\indent{\it Proof only if}. \ignorespaces
This is a direct consequence of Remark \ref{rem_jacCompOrth}, of the facts that if $\tilde{P}[m,n]$ has a solution, $\tilde{P}[m-1,n-1]$ must have one, and that the only parallelizable spheres are $\mathbb{S}^1$, $\mathbb{S}^3$ and $\mathbb{S}^7$ (see \cite{BotMil})
 and of
\begin{theorem}[{\cite[\S 1.7 p. 126]{Eck}}]	
For $m>n$, there exists a continuous function $\vartau \in\RR^{m \, \times \, n} \mapsto \tilde 
\gamma_1(\vartau)\in \RR^m$ with non zero values 
and satisfying
\\[0.6em]\null \hfill $\displaystyle 
\tilde \gamma_1(\vartau)^T \vartau=0
\qquad \forall \vartau\in\RR^{m \, \times \, n}:\,  
\textsf{Rank}(\vartau)=n
$\hfill \null \\[0.6em]
if and only if  $(m,n)$ is in 
one of the following $4$ pairs
\\[0.7em]
\null \hfill $(\geq 2,m-1)$
\hfill or\hfill 
$(\textrm{even},1)$
\hfill or\hfill 
$(7,2)$
\hfill or\hfill 
$(8,3)$\hfill \null 
\vbox{\hrule height0.6pt\hbox{%
\vrule height1.3ex width0.6pt\hskip0.8ex
\vrule width0.6pt}\hrule height0.6pt
}\\
\end{theorem}
A detailed version of the proof can be found in \cite{BerPraAndSIAMHal}.
}\stopjournal

\startlongue{
\indent{\it Proof only if}. \ignorespaces
See Appendix \ref{app_EckmannTilde}.
\hfill \vbox{\hrule height0.6pt\hbox{%
\vrule height1.3ex width0.6pt\hskip0.8ex
\vrule width0.6pt}\hrule height0.6pt
}
\\
}\stoplongue
\indent{\it Proof  if}. \ignorespaces
For $(m,n)$ equal to $(4,1)$ or $(8,1)$ respectively, possible solutions are
$$
\tilde\gamma(\vartau)
=
\left(
\begin{array}{@{}ccc@{}}
-\vartau_2 & \vartau_3 & \vartau_4
\\
\vartau _1 & -\vartau _4 & \vartau _3
\\
-\vartau _4 & -\vartau _1 & -\vartau _2
\\
\vartau _3 & \vartau _2 & -\vartau _1
\end{array}
\right) 
,\  
\tilde\gamma(\vartau)=\left(
\begin{array}{@{}ccccccc@{}}
\vartau_2 & \vartau_3 & \vartau_4    & \vartau _5 & \vartau_6 & \vartau_7 & \vartau_8
\\
-\vartau _1 & \vartau _4 & -\vartau _3 & \vartau _6 & -\vartau _5 & -\vartau _8 & \vartau _7
\\
-\vartau _4 & -\vartau _1 & \vartau _2 & \vartau _7 & \vartau _8 & -\vartau _5 & -\vartau _6 
\\
\vartau _3 & -\vartau _2 & -\vartau _1 & \vartau _8 & -\vartau _7 & \vartau _6 & -\vartau _5 
\\
-\vartau_6 & -\vartau_7 & -\vartau_8 & -\vartau _1 & \vartau_2 & \vartau_3 & \vartau_4
\\
\vartau _5 & -\vartau _8 & \vartau _7 & -\vartau _2 & -\vartau _1 & -\vartau _4 & \vartau _3
\\
\vartau _8 & \vartau _5 & -\vartau _6 & -\vartau _3 & \vartau _4 & -\vartau _1 & -\vartau _2
\\
-\vartau _7 & \vartau _6 & \vartau _5 & -\vartau _4 & -\vartau _3 & \vartau _2 & -\vartau _1
\end{array}
\right)
$$
where  $\vartau_j$ is the $j$th component of the vector $\vartau$. For $n=m-1$, we have the identity
\\[0.5em]\null \hfill $\displaystyle 
\det \left(
\vartau 
\quad
\tilde\gamma (\vartau) \right)= \sum_{j=1}^{m} \tilde\gamma_j(\vartau_{ij}) \, M_{j,m}(\vartau_{ij})
$\hfill \null \\[0.5em]
where $\tilde\gamma_j$ is the $j$th component of the vector-valued function $\tilde\gamma$ and the $M_{j,m}$, being the cofactors of $\left(
\vartau \  \tilde\gamma (\vartau) \right)$ computed along the last column, are polynomials in the given components $\vartau_{ij}$.
At least one of the $M_{j,m}$ is 
non-zero (because they are minors of dimension $n$ of $\vartau$ which is full-rank). So it is sufficient to take $\tilde\gamma_j(\vartau_{ij})=M_{j,m}(\vartau_{ij})$.
\hfill \vbox{\hrule height0.6pt\hbox{%
\vrule height1.3ex width0.6pt\hskip0.8ex
\vrule width0.6pt}\hrule height0.6pt
}

In the following example we show how by exploiting some structure we can reduce the problem to one of these 
$3$ pairs.

%have a
%universal explicit solution to the Jacobian complementation
%\jacobcomplet% a Jacobian completion
%problem.

\begin{example}[Continuation of Example \ref{ex_osci_1}]\normalfont
\label{ex_osci_2}
In Example \ref{ex_osci_1}, we have complemented
 the Jacobian 
(\ref{LP7}) with the gradient of a submersion and  observed that 
the components of this gradient are actually cofactors. We now know that 
this is consistent with the case $n=m-1$.
But we can also take advantage from the upper triangularity of the 
Jacobian (\ref{LP7}) and complement
only the vector $(-x_1,-x_2)$ by for instance $(x_2,-x_1)$. 
%Actually 
%in doing so, we move to the case with $n=1$ and $m=2$. 
The 
corresponding vector $\gamma $ is
$
\gamma(x) = (0,0,x_2,-x_1).
$
Here again, with 
Lemma
 \ref{lem_CS_P1}, we know that,
for any bounded open set $\Ouvs$ such that
$\CL(\OuvEx)\subset \Ouvs\subset\CL(\Ouvs)\subset\Ouvt$
the function
\\[0.1em]\null \hfill $\displaystyle 
\ohia(x,w)\;=\; \left(x_1\,  ,\: x_2\,  ,\: -x_1x_3 + x_2 w\,  ,\:
-x_2x_3-x_1 w\right)
$\hfill \null \\[0.7em]
is a diffeomorphism on $\Ouvs\times\BR_\epsilon(0)$.  In fact, in this particular case $\varepsilon$ can be
arbitrary, no need for it to be small.  However, the singularity at $\hat{x}_1=\hat{x}_2=0$ remains and
equation (\ref{n9})  is still not satisfied.
\hfill$\triangle$\end{example}

Given the very small number of cases where a universal formula exists, we now
look for a more general solution to the Jacobian complementation problem.

\subsection{Wazewski theorem}
\label{jac_comp_waz}

Historically, the Jacobian 
complementation problem
\jacobcomplet% a Jacobian completion
was first addressed by Wazewski (see \cite{Waz}).
His formulation was~:
\\ 
\textit{Given $m n$ continuous functions $\vartau_{ij} : \Ouvt\subset\RR^n \rightarrow \RR$,
look for $m(m-n)$ continuous functions $\gamma _{kl} : \Ouvt \to \RR$ such
that the following matrix is invertible for all $x$
in $\Ouvt$~:}
\begin{equation}
\label{eqWaz}
P(x)\;=\; \left(\begin{array}{@{}cc@{}}\vartau (x) & \gamma (x)
\end{array}\right)
\  .
\end{equation}
The difference with the previous section, is that here, we look for continuous functions $\gamma$ of $x$ in $\RR^n$ instead of continuous functions $\gamma$ of $\vartau$ in $\RR^{m\times n}$.
%The difference with the previous section, is that here, we look for continuous
%functions of the argument $x$ of $\vartau(x)$ instead of continuous functions of $\vartau$ itself.

Wazewski established that this
other version of the
 problem admits a far more general solution ~:

\begin{theorem}[{\cite[Theorems 1 and 3]{Waz} and
    \cite[page 127]{Eck}}]\label{theo_Waz}
If $\Ouvt$, equipped with the subspace topology of
$\RR^n$, is a contractible space, then there exists a
$C^\infty $
 function
$\gamma$ making  the matrix $P(x)$ in (\ref{eqWaz}) invertible for all $x$ in $\Ouvt$.
\end{theorem}

The reader is referred to \cite[page 127]{Eck} or \cite[pages 406-407]{Dug} and to
\cite[Theorems 1 and 3]{Waz}
for the complete proof of existence of a continuous function $\gamma $.
\startjournal{
It can be made smoother by using a partition of unity (see \cite{BerPraAndSIAMHal}).
}\stopjournal
We give the main constructive points of this proof 
below.
\startlongue{
Also, in appendix \ref{sec4}, we show, by using a partition of unity, how this continuous function $\gamma$ 
making $P$ invertible can be modified into a smoother one giving the 
same invertibility property.
}\stoplongue
But before this,
let us give the following corollary obtained as a 
consequence of Lemma  \ref{lem_CS_P1}.
\begin{corollary}
Let $\OuvEx $ be a bounded set, $\Ouvt$ be an open subset of $\RR^n$ containing $\CL(\OuvEx)$
and which, equipped with the subspace topology of
$\RR^n$, is a contractible space. Let also
$\ohi:\Ouvt\rightarrow \ohi(\Ouvt)\subset\RR^m$ be an injective immersion.
There exists a $C^1$ function $\gamma$ such that, for any
bounded open set $\Ouvs$ satisfying
$$
\CL(\OuvEx) \subset \Ouvs \subset \CL(\Ouvs) \subset \Ouvt
$$
we can find a strictly positive real number $\varepsilon $ such that
 the pair $(\ohia,\Ouvxw)$ defined in
(\ref{ohie}) solves Problem \ref{defP1}.
\end{corollary}

\textit{About the construction of $\gamma $:}
The proof of Theorem \ref{theo_Waz} given by Wazevski
is based on Remark \ref{rem_jacCompOrth},
% i.e. that if
%$\vartau(x)$ is a full rank $m \times n$ matrix, then 
%the complemented matrix
%$
%\left(\vartau(x) \ \gamma(x) \right)
%$
%is invertible whenever  $\gamma(x)$ is a full-rank $m \times (m-n)$ 
%matrix  satisfying
%\begin{equation}
%\vartau(x) ^T \gamma(x) = 0
%\  .
%\label{eq_waz_ortho}
%\end{equation}
 noting 
that, if we have the decomposition
$$
\vartau(x) = \left(
\begin{array}{c}
	A(x) \\
	B(x)
\end{array}
\right)
$$
with $A(x)$ invertible on
some given subset $\R$ of
$\Ouvt$, then 
$$
\gamma(x)= \left(
\begin{array}{c}
	C(x) \\
	D(x)
\end{array}
\right)
$$
satisfies \eqref{3}
 on $\R$ if and only if
$D(x)$ is invertible on $\R$ and we have
\begin{equation}
\label{eq_C_waz}
C(x) = - ( A^T(x) )^{-1} B(x) ^T D(x)\qquad \forall x\in\R\  .
\end{equation}
Thus, $C$ is imposed by the choice of $D$ and choosing $D$ invertible is enough to build $\gamma$ on $\R$.

Also, if we already have a candidate
$$
P(x)=\left(
\begin{array}{cc}
	A(x) & C_0(x) \\
	B(x) & D_0(x) 
\end{array}
\right)
$$
on a boundary $\partial\R$ of $\R$, then, necessarily, if $A(x)$ is 
invertible for
all $x$ in $\partial \R$, then
$D_0(x)$ is invertible and
$C_0(x) = - ( A^T(x) )^{-1} B(x) ^T D_0(x)$ all $x$ in $\partial \R$.  Thus,
to extend the construction of a continuous function $\gamma $ inside $\R$
from its knowledge on the boundary $\partial\R$,
it suffices to
pick $D$ as any invertible matrix satisfying
 $D=D_0$ on $\partial\R$.
Because we
can propagate continuously $\gamma $ from one boundary to the other,
Wazewski deduces from these two observations that, it is
sufficient to partition the set $\Ouvt$ into adjacent sets $\R_i$
where a given $n \times n$ minor $A_i$ is invertible.  This is
possible since $\vartau$
is full-rank on $\Ouvt$. When $\Ouvt$ is a parallelepiped, he shows that there
exists an ordering of the $\R_i$ such that the continuity of each $D_i$ can be
successively ensured. We illustrate this
construction  in Example \ref{ex-Waz} below.

\par\vspace{1em}
\begin{example}\normalfont
\label{ex-Waz}
Consider the function
$$
\vartau(x) \; = \;
\left(\begin{array}{ccc}
    1    &    0     &   0
\\
    0    &    1     &   0
\\
   -x_3   &    0     &  -x_1
\\
    0    &   -x_3    &  -x_2
\\
\frac{\partial \wp }{\partial x_1} {{x_3}}&
\frac{\partial \wp }{\partial x_2} {{x_3}}& \wp 
\end{array}\right) 
\quad ,\qquad 
\wp (x_1,x_2)\;=\; \max\left\{ 0 ,  \frac{1}{r^2} - (x_1^2 + x_2^2)  
\right\}^4\, .
$$
$\vartau(x)$ has full rank 3 for any $x$ in $\RR^3$, since $\wp (x_1,x_2) \neq 0$ when $x_1=x_2=0$.
To follow Wazewski's 
construction, let $\delta$ be a strictly positive real number and consider the following
$5$ regions of $\RR^3$ (see Figure \ref{fig_waz})
\begin{eqnarray*}
&\displaystyle 
\R_1\;=\; ]-\infty,-\delta]\times\RR^2
\quad ,\qquad 
\R_2\;=\; [-\delta,\delta]\times[\delta,+\infty]\times\RR,
\\
&\displaystyle 
\R_3\;=\; [-\delta,\delta]^2\times\RR
\quad ,\qquad 
\R_4\;=\; [-\delta,\delta]\times[-\infty,-\delta]\times\RR
\quad ,\qquad
\R_5\;=\; [\delta,+\infty[\times\RR^2.
\end{eqnarray*}
 We select $\delta$ sufficiently small in such a way that 
 $\wp$ is not $0$ in $\R_3$.
 
\begin{figure}[h]
  	\centering
  	\def\svgwidth{6cm}
  	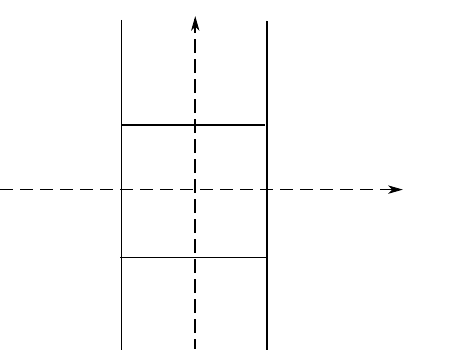  	
  	\caption{ Projections of the regions $\R_i$ on $\RR^2$.}
  	\label{fig_waz}
\end{figure}

We start Wazewski's algorithm in 
 $\R_3$. Here, the invertible minor $A$ is given by rows $1$, $2$ and $5$ of 
 $\vartau$ (full-rank lines of $\vartau$) and $B$ by rows $3$ and $4$. With picking $D$ as the 
 identity, $C$ is $(A^T)^{-1}B$ according to \eqref{eq_C_waz}. $D$ gives rows $3$ and $4$ of 
 $\gamma $ and $C$ gives rows $1$, $2$ and $5$ of $\gamma$.
\\
Then we move to the region $\R_2$.  There the 
matrix $A$ is given by rows $1$, $2$ and $4$ of $\vartau$, $B$ by rows
$3$ and $5$. Also $D$, along the boundary between $\R_3$ and $\R_2$, is 
given by rows $3$ and $5$ of $\gamma$ obtained in the previous step.
We extrapolate this inside $\R_2$ by keeping $D$ constant in planes 
$x_1=$constant. An expression for $C$ and therefore for $\gamma $ 
follows.
\\
We do exactly the same thing for $\R_4$.
\\
Then we move to the region $\R_1$.  There the 
matrix $A$ is given by rows $1$, $2$ and $3$ of $\vartau$, $B$ by rows
$4$ and $5$. Also $D$, along the boundary between $\R_1$ and $\R_2$, 
between $\R_1$ and $\R_3$ and between $\R_1$ and $\R_4$, is 
given by rows $4$ and $5$ of $\gamma$ obtained in the previous steps.
We extrapolate this inside $\R_1$ by keeping $D$ constant in planes 
$x_2=$constant. An expression for $C$ and therefore for $\gamma $ 
follows.
\\
We do exactly the same thing for $\R_5$.

Note that this construction produces a continuous $\gamma$, but we could have extrapolated
$D$ in a smoother way to obtain $\gamma$ as smooth as necessary.
\hfill$\triangle$\end{example}

Although Wazewski's method provides a more general answer to the problem of Jacobian complementation than the few solvable $\tilde{P}[m,n]$ problems, the explicit expressions of $\gamma$ given in Section \ref{jac_comp_priv}  are preferred in practice (when the couple $(m,n)$ is appropriate) to Wazewski's costly computations.

%%%%%%%%%%%%%%%%%%%%%%%%%%%%%%%%%%%%%%%%%%%%%%%%%%%%%%%%%%%%%%%%%%%%%%%%%%
%%%%%%%%%%%%%%%%%%%%%%%%%%%%%%%%%%%%%%%%%%%%%%%%%%%%%%%%%%%%%%%%%%%%%%%%%%

\section{About Problem \ref{defP2} : Image extension of a diffeomorphism}~
\label{diffeo_ext}
We study now how a diffeomorphism can be augmented to make its image be the whole set 
$\RR^m$, i.e. to make it surjective.

\subsection{A sufficient condition}

There is a rich literature reporting very advanced results on the 
diffeomorphism extension problem.
In the following some
of the techniques are inspired from \cite[Chapter 8]{Hir}
and \cite[pages 2,  7 to 14 and 16 to 18]{Milnor}(among others). 
Here we are interested in the 
 particular aspect of this topic which is the diffeomorphism image extension as described by
Problem \ref{defP2}.
A very first necessary condition about this problem is in the following remark.

\begin{remark}
	\normalfont
	\label{rem_contractible_diffeoExt}
Since $\ohie$, obtained solving Problem \ref{defP2}, makes the set $\OuvDef$ diffeomorphic to $\RR^m$, 
$\OuvDef$ must be contractible.
\end{remark}

One of the key technical property which will allow us to solve Problem \ref{defP2} can be phrased as follows.

\begin{definition}[Condition $\condC$]
\label{def2}
An open subset $E$ of $\RR^m$ is said to verify condition $\condC$ if
there exist a $C^1$ function $\kappa:\RR^m\to \RR$, a bounded\footnote{If not
replace $\chi$ by $\frac{\chi}{\sqrt{1+|\chi|^2}}$.} $C^1$
 vector
field $\chi$, and a closed set $K_0$
 contained in $E$ such that:
\begin{enumerate}
\item
\label{point1}
 $E=\left\{ \varchi\in \RR^n, \kappa(\varchi) <0 \right\}$
\item
$K_0$ is globally attractive  for $\chi$
\item
\label{point2}
we have the following transversality property:
$$
\frac{\partial \kappa}{\partial \varchi}(\varchi) \chi(\varchi) <0
\qquad \forall \varchi\in \RR^m:\:  \kappa(\varchi)=0.
$$
\end{enumerate}
\end{definition}

The two main ingredients of this condition are the function $\kappa $ 
and the vector field $\chi$ which, both, have to satisfy the 
transversality property $\condC$.\ref{point2}.
In the case where only the function $\kappa$ is given 
satisfying $\condC$.\ref{point1} and with no critical point on the 
boundary of $E$, its gradient could play the role of $\chi$. But then 
for $K_0$ to be globally attractive we need at least to remove all the 
possible critical points that $\kappa $ could have outside $K_0$. This 
task is performed for example on Morse functions in the proof of the 
$h$-Cobordism Theorem \cite{Milnor}. We are in a much simpler situation when $\chi$ is given and makes $E$ forward 
invariant.

\begin{lemma}
\label{lemCond}
Let $E$ be a bounded open subset  of $\RR^m$, 
$\chi$ be a bounded $C^1$ vector field , and $K_0$ be a compact set contained in
$E$ such that:
\begin{enumerate}
\item
 $K_0$ is globally asymptotically stable  for $\chi$
\item
$E$ is forward invariant for $\chi$.
\end{enumerate}
For any strictly positive real number $\overline{d}$, 
there exists a bounded set  $\E$ such that
$$
\CL(E)\subset \E \subset \{\varchi\in\RR^m, \  \inf_{z_E \in E}|z-z_E| \leq \overline{d} 
%d(\varchi ,E) \in [0, \overline{d}]
\}
$$
and $\E$ verifies condition $\condC$.
\end{lemma}

This Lemma says roughly that if $E$
does not satisfy conditions $\condC$.\ref{point1} or $\condC$.\ref{point2} but
is forward invariant for 
$\chi$, then Condition $\condC$ is satisfied by an arbitrarily close 
superset of $E$. Its proof is given in Appendix \ref{ann_lem3}.

Our main result on the diffeomorphism image extension problem is:

\par\vspace{0.5em}\noindent

\begin{theorem}[Image extension]
\label{theo_diff_ext}
Let $\OuvDef $ be an open subset of $\RR^m $
and $\ohia$: $\OuvDef  \to  \RR^m$
be a diffeomorphism.
If 
\begin{list}{}{%
\parskip 0pt plus 0pt minus 0pt%
\topsep 0pt plus 0pt minus 0pt%\topsep 0.5ex plus 0pt minus 0pt%
\parsep 0pt plus 0pt minus 0pt%
\partopsep 0pt plus 0pt minus 0pt%
\itemsep 0pt plus 0pt minus 0pt%\itemsep 0.5ex plus 0pt minus 0pt
\settowidth{\labelwidth}{\quad a)}%
\setlength{\labelsep}{0.5em}%
\setlength{\leftmargin}{\labelwidth}%
\addtolength{\leftmargin}{\labelsep}%
}
\item[a)]
either $\ohia(\OuvDef)$ verifies condition $\condC$,
\item[b)]
or $\OuvDef$  is $C^2$-diffeomorphic to $\RR^m$ and $\ohia$ is $C^2$, 
\end{list}
then for any compact set
$K$ in $\OuvDef$, there exists a
diffeomorphism $\ohie : \OuvDef \rightarrow \RR^m$ 
solving Problem \ref{defP2}.
\end{theorem}

The proof of case a) of this theorem is given in Section \ref{sec3}. It provides  an explicit construction of $\ohie$. The proof of case b) can be found in Appendix \ref{app_diff_ext_case_b}. For the time being, we observe that a direct consequence is :

\begin{corollary}
\label{cor_theo_diff_ext}
Let $\OuvEx$ be a bounded subset of $\RR^n$,
$\Ouvxw$ be an open subset of $\RR^m$ containing $\CL(\OuvEx\times\{0\})$ and
 $\ohia:\Ouvxw\rightarrow\ohia(\Ouvxw)$  be a diffeomorphism
such that
\begin{list}{}{%
\parskip 0pt plus 0pt minus 0pt%
\topsep 0pt plus 0pt minus 0pt%\topsep 0.5ex plus 0pt minus 0pt%
\parsep 0pt plus 0pt minus 0pt%
\partopsep 0pt plus 0pt minus 0pt%
\itemsep 0pt plus 0pt minus 0pt%\itemsep 0.5ex plus 0pt minus 0pt
\settowidth{\labelwidth}{\quad a)}%
\setlength{\labelsep}{0.5em}%
\setlength{\leftmargin}{\labelwidth}%
\addtolength{\leftmargin}{\labelsep}%
}
\item[a)]
either $\ohia(\Ouvxw)$  verifies condition $\condC$,
\item[b)]
or
$\Ouvxw$ is $C^2$-diffeomorphic to $\RR^m$ and $\ohia$ is $C^2$.
\end{list}
Then, there exists a diffeomorphism
$\ohie:\Ouvxw \to \RR^m$, such that 
$$
\ohie(\Ouvxw)=\RR^m \quad , \quad \ohie(x,0)\;=\; \ohia(x,0)\qquad \forall x\in\OuvEx \ . 
$$
Thus, if besides the pair $(\ohia,\Ouvxw)$ solves Problem \ref{defP1}, then $(\ohie,\Ouvxw)$ solves Problems  \ref{defP1} and
\ref{defP2}.
\end{corollary}

\subsection{Proof of part a) of Theorem \ref{theo_diff_ext}}
\label{sec3}
We have the following technical lemma a constructive proof of which
is  given in Appendix \ref{ann_lem1}.
\begin{lemma}
	\label{lem1}
	Let $E$ be an open strict subset 
	of $\RR^m$ verifying Condition $\condC$. For any closed subset $K$ 
	of $E$, lying at a strictly positive distance of the boundary of $E$, 
	there exists a 
	diffeomorphism $\phi$: $\RR^m\rightarrow E$, such 
	that $\phi$ is the identity function on $K$.
\end{lemma}
In the case a) of Theorem \ref{theo_diff_ext}, we suppose that $\ohia(\OuvDef)$ satisfies $\condC$.
Now, $\ohia$ being a diffeomorphism on an open set
$\OuvDef$, the image of any compact subset $K$ of $\OuvDef$ is a compact subset of $\ohia(\OuvDef)$.
According to Lemma \ref{lem1}, there exists a diffeomorphism $\phi$ from $\RR^m$ to $\ohia(\OuvDef)$ which is
the identity on $\ohia(K)$.  Thus, the function 
$
% \label{5}
\ohie\;=\; \phi^{-1}\circ \ohia
$
solves Problem \ref{defP2} and the theorem is proved.

\begin{example}[Continuation of Example \ref{ex_osci_1}]\normalfont
\label{ex_osci_4}
In Example \ref{ex_osci_1}, we have introduced the function
$$
F(\xi)=\xi_2 \xi_3-\xi_1 \xi_4 \; \triangleq \; \frac{1}{2} \xi^T M \xi
$$
as a submersion on $\RR^4\!\setminus\!\{0\}$ satisfying
\begin{equation}
\label{LP8}
F(\ohi(x))=0,
\end{equation}
where $\ohi$ is the injective immersion  given in (\ref{eq_ExplHGImm}).
With it we have augmented $\ohi$ as
$$
\ohia(x,w)=\ohi(x)+\frac{\partial F}{\partial \xi}^T(\ohi(x)) \, w =\ohi(x)+ M \ohi(x) \, w
$$
which is a diffeomorphism on $\Ouvxw=\Ouvs\times ]-\varepsilon ,\varepsilon [$
for some 
strictly positive real number $\varepsilon $.

To modify $\ohia$ in $\ohie$ satisfying $\ohie(\Ouvxw)=\RR^4$, we let $K$ be the compact set
$$
K =\CL(\ohia(\OuvEx\times\{0\}))\  \subset \ohia(\Ouvxw)\subset \RR^4
\  .
$$
With Lemma \ref{lem1}, we know that, if
$\ohia(\Ouvxw)$ verifies condition $\condC$,  there exists a diffeomorphism
$\phi$ defined on $\RR^4$ such that $\phi$ is the identity function on the compact set $ K$ and
$\phi(\RR^4)=\ohie(\Ouvxw)$.  In that case, as seen above,
the diffeomorphism $\ohie=\phi^{-1}\circ\ohia$ defined
on $\Ouvxw$ is such that $\ohie=\ohia$ on $\OuvEx\times\{0\}$ and
$\ohie(\Ouvxw)=\RR^4$, i.e. would be a solution to Problems \ref{defP1} and \ref{defP2}.
Unfortunately this is impossible. Indeed, due to the observability singularity at
$x_1=x_2=0$, $\Ouvs$ (and thus $\Ouvxw$) is not contractible.  Therefore, there is no 
diffeomorphism $\ohie$ such that $\ohie(\Ouvxw)=\RR^4$.  We will see
in Section \ref{disc_P3} how this problem can be overcome.
For the time being, we show that it is still possible to find $\ohie$ such that
$\ohie(\Ouvxw)$ covers ``almost all'' $\RR^4$.  
The
idea is to find an approximation $E$ of $\ohia(\Ouvxw)$ verifying condition $\condC$ and apply the same method
on $E$.  Indeed, if $E$ is close enough to $\ohia(\Ouvxw)$, one can expect to have $\ohie(\Ouvxw)$
``almost equal to'' $\RR^4$.

With (\ref{LP8}) and since $M^2=I$, we have, 
$
F(\ohia(x,w))= |\ohi(x) |^2\, w  
$.
Since $\Ouvxw$ is bounded, there exists $\delta >0$ such that the set
$
E = \left\{ \xi\in \RR^4 :\,   F(\xi)^2<\delta \right \}
$
contains $\ohia(\Ouvxw)$ and thus the compact set $K$. Let us show that $E$ verifies condition $\condC$. We pick
\\[-0.1em]\null \hfill $\displaystyle 
\kappa (\xi)\;=\; F(\xi)^2-\delta\;=\; 
\left(\frac{1}{2}
\xi  ^T M \xi 
\right)^2-\delta
\  .
$\hfill \null \\[0.7em]
and consider the vector field $\chi$ \\[1em] \null \hfill $\displaystyle \chi (\xi) =-2 \frac{\partial \kappa
}{\partial \xi}(\xi)=-[\xi ^T M \xi ]\, M\xi $\hfill or more simply \hfill $\displaystyle \chi(\xi)=-\xi \ .
$ \hfill \null \\[1em] The latter implies the transversality property $\condC$.\ref{point2} is
verified.  Besides, the closed set $K_0=\{0\}$ is contained in $E$ and is globally attractive for the vector
field $\chi$.

Then Lemma \ref{lem1} gives the existence of a diffeomorphism 
$\phi: \RR^4 \to E$ which is the identity on $ K$ and verifies $\phi(\RR^4)=E$.
We obtain an expression of $\phi$ by following
the constructive proof of this Lemma (see Appendix \ref{ann_lem1}).
Let $E_\varepsilon$ be the set
$$
E_\varepsilon = \left\{ 
\xi \in \RR^4:\,   \left(\frac{1}{2} \xi ^T M \xi
\right)^2 < e^{-4\varepsilon} \, \delta \right\} \ .
$$
It contains $ K$. Let also $\temps:[-\varepsilon,+\infty[\to \RR$ and $\tde :\RR^4\setminus E_\varepsilon \to \RR$ be
the 
functions  defined as
\\[0.1em]\null \hfill $\displaystyle 
\temps(t)\;=\; \frac{(t+\varepsilon )^2}{2\varepsilon +t}
\quad ,\qquad 
\tde (\xi)
=\frac{1}{4}\ln \frac{\left(\frac{1}{2} \xi ^T M \xi
	\right)^2}{\delta} \ .
$\refstepcounter{equation}\label{eq_def_nu}\hfill$(\theequation)$\\[0.7em]
% $\temps$ is a function such that $\temps-t$ is a bijection from $[-\varepsilon,+\infty[$ to $]0,\varepsilon]$, and
$\tde (\xi )$ is the time that a solution of
$
\dot \xi =\chi(\xi)=-\xi
$
with initial condition  
$\xi $ needs to reach the boundary of $E$
i.e. $e^{-\tde (\xi )} \xi $ belongs to the boundary of $E$.
From the proof Lemma \ref{lem1}, we know
the
function
$\phi$ : $\RR^4\rightarrow E$ defined as :
\begin{equation}
\label{LP10}
\phi(\xi ) = 
\left\{
\begin{array}{ll}
\xi 
&, \quad \textrm{if} \ \left(\frac{1}{2} \xi ^T M \xi
\right)^2 
\leq e^{-4\varepsilon} \delta ,\\[0.3em]
e^{-\temps(\tde (\xi ))} \xi 
&, \quad \textrm{otherwise},
\end{array}
\right.
\end{equation}
is a diffeomorphism 
$\phi: \RR^4 \to E$ which is the identity on $ K$ and verifies $\phi(\RR^4)=E$.
% The way $\phi$ maps $\RR^4$ onto $E$ is depicted on Figure \ref{fig_ext_diffeo}.
% \begin{figure}[h]
% 	\centering
% 	\def\svgwidth{9cm}
% 	\input{dessin_ext_diffeo.pdf_tex}%
% 	\caption{%
% Transformation of $\displaystyle s=\frac{\left(\frac{1}{2}\xi ^T M \xi \right)^2}{\delta}$ when mapping $\xi$
% in $\RR^4$ to $\phi(\xi)$ in $E$.}
% 	\label{fig_ext_diffeo}
% \end{figure}

As explained above, we use $\phi$ to replace $\ohia$ by the diffeomorphism
$\ohie=\phi^{-1}\circ\ohia$ also defined on $\Ouvxw$.  But, because $\ohia(\Ouvxw)$ is a strict 
subset of $E$, $\ohie(\Ouvxw)$ is a strict subset of $\RR^4$, i.e. equation (\ref{n9})  is not satisfied.  
Nevertheless, for any trajectory of the observer $t\mapsto \hat{\xi}(t)$ in $\RR^4$, our estimate
defined by $(\hat{x},\hat{w})=\ohie^{-1}(\hat{\xi})$ will be such that $\ohia(\hat{x},\hat{w})$
remains in $E$, along this trajectory i.e. $|\ohi(\hat x) |^2\,\hat w <\delta$.  This ensures that, far from the observability singularity where
$|\ohi(\hat x) |=0$, $\hat{w}$ remains sufficiently small to keep the invertibility of the
Jacobian of $\ohie$.  But we still have a problem close to the observability singularity, i.e.  when  $(\hat x_1,\hat 
x_2)$ is close to the origin.  We shall see in Section \ref{disc_P3} how to avoid this difficulty via a better
choice of the initial injective immersion $\ohi$.
\hfill$\triangle$\end{example}

\subsection{Application : bioreactor}
\label{bioreactor}
As a more practical illustration we consider the model of
bioreactor presented in \cite{GauHamOth} :
$$
\dot{x}_1=\frac{a_1x_1x_2}{a_2x_1+x_2}-ux_1 \ , \ 
\dot{x}_2=-\frac{a_3a_1x_1x_2}{a_2x_1+x_2}-ux_2+ua_4 \ ,\
y=x_1
$$
where the $a_i$'s are strictly positive real numbers and the control $u$ verifies :
$
0<u_{min}<u(t)<u_{max}<a_1
$.
This system evolves in the set
$
\Ouvt\;=\; \left\{ x\in\RR^2\,  :\: x_1>\varepsilon_1\: ,\; x_2>-a_2x_1 \right \}
$
which is forward invariant.
A high gain observer design leads us to consider the function 
$\ohi:\Ouvt\to \RR^2$ defined as :
\\[0.5em]\null \hfill $\displaystyle 
\ohi(x_1,x_2) \;=\;  \left( x_1,\left.\dot x_1\right|_{u=0}\right)\;=\; \left( x_1 , \frac{a_1 x_1 x_2}{a_2x_1+x_2} \right)
\  .
$\hfill \null \\[0.3em]
It is a diffeomorphism onto
$$
\ohi(\Ouvt)\;=\; \left\{ \xi\in\RR^2\,  :\: \xi_1>0\: ,\; 
a_1\xi_1>\xi_2 \right \} \ .
$$
The image by $\ohi$ of the bioreactor dynamics is of the form
$$
\dot{\xi}_1 \;=\;  \xi_2 + g_1(\xi_1) u \quad ,\qquad 
\dot{\xi}_2 \;=\;  \of _2(\xi_1,\xi_2) + g_2(\xi_1,\xi_2) u
$$
for which the following high gain observer can be built:
\begin{equation}
\label{bio_dyn_obs}
\dot \ox_1=\ox_2 + g_1(\ox_1) u - k_1 \ell (\ox_1-y)
\quad ,\quad 
\dot \ox_2=\of_2 (\ox_1,\ox_2) + g_2(\ox_1,\ox_2) u  - k_2 \ell (\ox_1-y)
\: ,
\end{equation}
where $k_1$ and $k_2$ are strictly positive real numbers 
and $\ell$ sufficiently large. As in \cite{GauHamOth}, $\ohi$ being a diffeomorphism
the dynamics of this observer in the $x$-coordinates are
\begin{equation}
\label{bio_dyn_obs_x}
\renewcommand{\arraystretch}{1.5}
\dot{\hat{x}}\;=\; \left(\begin{array}{c}
\frac{a_1\hat x_1\hat x_2}{a_2\hat x_1+\hat x_2}-u\hat x_1 \\
-\frac{a_3a_1\hat x_1\hat x_2}{a_2\hat x_1+\hat x_2}-u\hat x_2+ua_4 
\end{array}\right)
+ \ell
\left(
\begin{array}{cc}
1 & 0
\\
-1
&
\frac{(a_2\hat x_1+\hat x_2)^2}{a_1a_2\hat x_1^2 }
\end{array}\right)  \left(\begin{array}{c}
k_1 \\ k_2
\end{array}\right) (\ox_1-y)
\  .
\end{equation}
Unfortunately the right hand side is singular at $\hat x_1=0$ 
and $\hat x_2 = -a_1 \hat x_1$. $\Ouvt$ being forward invariant, the system trajectories stay away from the 
singularity. But nothing guarantees the same property holds for the observer trajectories given by (\ref{bio_dyn_obs_x}). 
In other words, since $\ohi$ is already a diffeomorphism, Problem \ref{defP1} is solved with $m=n$, 
$\ohia=\ohi$ and $\Ouvxw=\Ouvt$.
But (\ref{n9}) is not satisfied, i.e.  Problem \ref{defP2} must be solved.

To construct the extension $\ohie$ of $\ohia$, we view the image 
$\ohia(\Ouvxw)$ as the intersection
$
\ohia(\Ouvxw)= E_1 \cap E_2 
$
with :
$$
E_1\;=\; \left\{ (\xi_1,\xi_2) \in \RR^2, \ \xi_1> \varepsilon_1\right\}
\quad ,\qquad 
E_2\;=\; \left\{(\xi_1,\xi_2) \in \RR^2, \ a_1\xi_1 >\xi_2 \right\} \ .
$$
This exhibits the fact that $\ohia(\Ouvxw)$ does not satisfy the 
condition $\condC$ since its boundary is not $C^1$. We could 
smoothen this boundary to remove its ``corner''. But we prefer to 
exploit its particular ``shape'' and proceed as follows : 
\begin{list}{}{%
\parskip 0pt plus 0pt minus 0pt%
\topsep 0.5ex plus 0pt minus 0pt%
\parsep 0pt plus 0pt minus 0pt%
\partopsep 0pt plus 0pt minus 0pt%
\itemsep 0.5ex plus 0pt minus 0pt
\settowidth{\labelwidth}{\  \  1.}%
\setlength{\labelsep}{0.5em}%
\setlength{\leftmargin}{\labelwidth}%
\addtolength{\leftmargin}{\labelsep}%
}
\item[1.]
We build a diffeomorphism $\phi_1$ : 
$\RR^2\rightarrow E_1$ which acts on $\xi_1$ without changing $\xi_2$.
\item[2.]
We build a diffeomorphism $\phi_2$ : 
$\RR^2\rightarrow E_2$ which acts on $\xi_2$ without changing $\xi_1$.
\item[3.]
Denoting $\phi=\phi_2\circ\phi_1:\RR^2\to E_1\cap E_2$,
we take $\ohie =	\phi^{-1} \circ \ohia:\,  \Ouvxw 
\to \RR^2 $.
\end{list}
To build $\phi_1$ and $\phi_2$, we follow the procedure given in the proof of 
Lemma \ref{lem1} since 
$E_1$ and $E_2$ satisfy condition $\condC$ with~:
$$
\kappa_1(\xi)=\varepsilon_1-\xi_1 \; ,\   \kappa_2(\xi)=\xi_2-a_1\xi_1
% $$
% and :
% $$
\; ,\  
\chi_1(\xi)=\!
\left(
\begin{array}{c}
-(\xi_1-1) \\
0
\end{array}
\right) 
\; ,\  
\chi_2(\xi)= \!
\left(
\begin{array}{c}
0 \\
-(\xi_2+1)
\end{array}
\right) .
$$
By following the same steps as in Example \ref{ex_osci_4}, with 
$\varepsilon $ an arbitrary small strictly positive real number
and $\temps $ defined in (\ref{eq_def_nu}), we obtain~:
\begin{equation}
\label{LP17}
\left|
\begin{array}{@{}r@{\; }c@{\; }l@{}}
\tde _1(\xi)&=&\ln \frac{1-\xi_1}{1-\varepsilon} 
\\[0.5em]
E_{\varepsilon,1}&=&\left\{ (\xi_1,\xi_2)\in \RR^2, \ \xi_1> 1-\frac{1-\varepsilon}{e^{\varepsilon} }\right\}
\\[0.5em]
\phi_1(\xi) &=&
\left\{
\begin{array}{@{\,  }l@{\; }l@{}}
\xi &, \quad \textrm{if} \ \xi\in E_{\varepsilon,1}\\
\frac{\xi_1-1}{e^{\temps (\tde _1(\xi))}}+1 &, \quad \textrm{otherwise}
\end{array}
\right.
\end{array}
\right.\quad  \left|
\begin{array}{@{}r@{\; }c@{\; }l@{}}
\tde _2(\xi)&=&\ln \frac{\xi_2+1}{a_1 \xi_1+1} \ ,
\\[0.5em]
E_{\varepsilon,2}&=&\left\{ (\xi_1,\xi_2)\in \RR^2, \ \xi_2 \leq \frac{a_1\xi_1+1}{e^{\varepsilon}}-1 \right\}
\\[0.5em]
\phi_2(\xi) &=&
\left\{
\begin{array}{@{\,  }l@{\; }l@{}}
\xi &, \quad \textrm{if} \ \xi\in E_{\varepsilon,2} \\
\frac{\xi_2+1}{e^{\temps (\tde _2(\xi))}}-1 &, \quad \textrm{otherwise}
\end{array}
\right.
\end{array}
\right.
\end{equation}
% \begin{equation}
% \label{LP17}
% \begin{array}{rcl}
% \tde _1(\xi)&=&\ln \frac{1-\xi_1}{1-\varepsilon} 
% \\[0.5em]
% E_{\varepsilon,1}&=&\left\{ (\xi_1,\xi_2)\in \RR^2, \ \xi_1> 1-(1-\varepsilon)e^{-\varepsilon} \right\}
% \\[0.5em]
% \phi_1(\xi) &=&
% \left\{
% \begin{array}{ll}
% \xi &, \quad \textrm{if} \ \xi\in E_{\varepsilon,1}\\
% e^{-\temps (\tde _1(\xi))}(\xi_1-1)+1 &, \quad \textrm{otherwise}
% \end{array}
% \right.
% \\[1em]\hline
% \vrule height 1.2em depth 0pt width 0pt
% \tde _2(\xi)&=&\ln \frac{\xi_2+1}{a_1 \xi_1+1} \ ,
% \\[0.5em]
% E_{\varepsilon,2}&=&\left\{ (\xi_1,\xi_2)\in \RR^2, \ \xi_2 \leq e^{-\varepsilon}(a_1\xi_1+1)-1 \right\}
% \\[0.5em]
% \phi_2(\xi) &=&
% \left\{
% \begin{array}{ll}
% \xi &, \quad \textrm{if} \ \xi\in E_{\varepsilon,2} \\
% e^{-\temps (\tde _2(\xi))}(\xi_2+1)-1 &, \quad \textrm{otherwise}
% \end{array}
% \right.
% \end{array}
% \end{equation}

We remind the reader that, in the $\ox$-coordinates, the observer dynamics are 
not modified. The difference between using $\ohi$ or $\ohie$ is seen in the $\hat x$-coordinates 
only. And, by construction it has no effect on the system trajectories since we have
$$
\ohi(x)=\ohie(x)\quad \forall x \in \Ouvt\; \mbox{``}-\varepsilon \mbox{''}
\  .
$$
As a consequence the difference between $\ohi$ 
and $\ohie$ is significant only during the transient, making sure, for the latter, that 
$\hat x$ never reaches a singularity of the Jacobian of $\ohie$.

We present in 
Figure \ref{bio_simu_traj} the results in the $\ox$ coordinates (to 
allow us to see the effects of both $\ohi$ and $\ohie$) of a simulation with 
(similar to \cite{GauHamOth}) :
\\[0.5em]\null \hfill $\displaystyle
\begin{array}{c}
\displaystyle 
a_1=a_2=a_3=1\; ,\  a_4=0.1
\\
\displaystyle 
u(t)=0.08\ \mbox{for}\  t\leq 10\  ,\quad   = 0.02 \  \mbox{for}\  10 
\leq t\leq 20\  ,\quad  = 0.08 \  \mbox{for}\  t\geq 
20
\\
\displaystyle 
x(0)=(0.04,0.07), \quad \hat{x}(0)=(0.03,0.09)
,\quad \ell =5.
\end{array}
$\hfill \null 
\begin{figure}[h]
	\centering
	\def\svgwidth{10cm}
	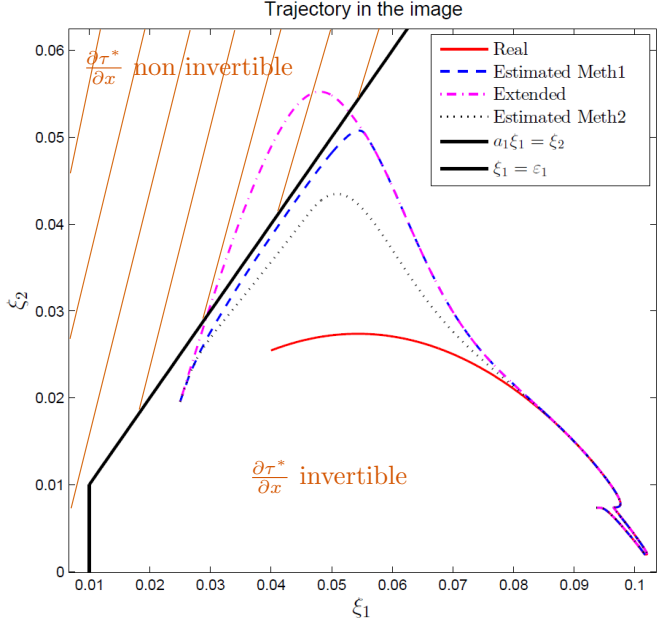
	\caption{Bioreactor and observers solutions in the $\ox$-coordinates}%
	\label{bio_simu_traj}
\end{figure}

%\begin{figure}%
%\includegraphics[scale=0.5]{traj_in_image_point.pdf}%
%\caption{Bioreactor and observers solutions in the $\ox$-coordinates}%
%\label{bio_simu_traj}%
%\end{figure}

The solid black curves are the singularity locus.
The red (= solid dark) curve represents the bioreactor solution.
The magenta (= light grey dashdot) curve represents the solution of the observer built with 
$\ohie$. It evolves freely in 
$\RR^2$ according to the dynamics (\ref{bio_dyn_obs}), not worried by 
any constraints. The blue (= dark dashed) curve represents its image 
by $\phi$ which 
brings it back inside the constrained domain where 
$\ohi^{-1}$ can then be used. This means these two curves represent the same 
object but viewed in different coordinates.

The solution  of the observer built with 
$\ohi$ would coincide with the magenta (= light grey dashdot) curve up to 
the point it reaches one solid black curve of a 
singularity locus. At that point it leaves $\ohi(\Ouvt)$ and consequently 
stop  existing in the $x$-coordinates.

As proposed in \cite{MagPas,AstPra},
instead of keeping the raw dynamics (\ref{bio_dyn_obs}) untouched as 
above, 
another solution would be to modify them to force $\ox$ to remain in the set $\ohi(\Ouvt)$. For instance, taking 
advantage of the convexity of this set, the modification 
proposed in \cite{AstPra} consists in adding to (\ref{bio_dyn_obs}) 
the term
\begin{equation}
\label{LP16}
\mathcal{M} (\ox)= -
g
\, S_\infty \,
\frac{\partial \mathfrak{h}}{\partial \ox}(\ox)^T \, \mathfrak{h}(\ox) 
\quad ,\qquad 
\mathfrak{h}(\ox) \;=\; \left(\begin{array}{c}
\max\{\kappa_1(\ox)+
\varepsilon 
,0 \}  ^2
\\
\max\{\kappa_2(\ox)+
\varepsilon 
,0 \} ^2
\end{array}\right)
\end{equation}
with
$S_\infty$ a symetric positive definite matrix depending on $(k_1,k_2,\ell)$,
$ \varepsilon $ an arbitrary small real number
and $
g
$ a sufficiently large real number.
The solution corresponding to this modified observer dynamics is 
shown in Figure \ref{bio_simu_traj} with
the dotted black curve. As expected it 
stays away from the the singularities locus in a very efficient way. 
But, for this method to apply, we have the restriction that 
$\ohi(\Ouvt)$ should be convex, instead of satisfying the less 
restrictive condition $\condC$. Moreover, to guarantee that $\ox$ is 
in $\ohi(\Ouvt)$, $
g
$ has to be large enough and even larger when the measurement noise is larger.
On the contrary, when the observer is built with 
$\ohie$, there is no need to tune properly any parameter to obtain 
convergence, at least theoretically. Nevertheless there maybe some 
numerical problems when $\ox$ becomes
too large or equivalently $\phi(\ox)$ is too close to the boundary of $\ohi(\Ouvt)$.
To overcome this difficulty we can select the ``thickness'' of 
the layer (parameter $\varepsilon$ in (\ref{LP17}))
sufficiently large. Actually instead of ``opposing'' the two methods, we 
suggest to combine them. The
modification (\ref{LP16}) makes sure $\ox$ does not go too far outside the domain, 
and $\ohie$ makes sure that $\hat x$ does not cross the 
singularity locus.

%%%%%%%%%%%%%%%%%%%%%%%%%%%%%%%%%%%%%%%%%%%%%%%%%%%%%%%%%%%%%%%%%%%%%%%%%%%%%%%%%%%%%%%%%%%%%%%%%%%%%%%%%%%%%%%%%%%%%%%%%%%%%%%%%%%%%%%%%%%%%%%%%%%%%%%%%%%%%%%%%%%%%%%%%%%%%%%%%%%%%%%%%%%%%%%%%%%%%

\section{About the requirement that  $(\ohede{x},\of  )$ is in $\setphitau$
in Proposition \ref{prop_CSobsWithoutInversion}
}~
\label{sec1}
\\
Throughout Sections \ref{jac_comp} and \ref{diffeo_ext}, we have given conditions
under which it is possible to solve 
 Problem \ref{defP1} and Problem \ref{defP2} when 
Assumption $\hypo$ holds and $\OuvEx$ is bounded.

However, to apply  Proposition \ref{prop_CSobsWithoutInversion} we need
$\ohede{x}$, the $x$-component of 
the inverse $\ohe$ of $\ohie$, solution of Problem \ref{defP2}, to be associated with a function $\of  $ such that the pair 
$(\of  ,\ohede{x})$ is in the set $\setphitau$ given by assumption 
$\hypo$.

Fortunately pairing a function $\of  $ with a function $\ohede{x}$ obtained from a left inverse
of $\ohie$  is not as difficult as it seems, at least for general purpose observer designs
such as high gain observers or  nonlinear Luenberger observers.

Indeed, we have already observed in point \ref{point3} of Remark \ref{rem1} that if, as for Luenberger observers, 
there is a pair, in the set 
$\setphitau$, the component $\of  $ of which does 
not depend on $\tau$, then we can associate this $\of  $ to any
$\ohede{x}$.

Also, for high gain observers, we need only that $\ohede{x}$, used as argument of $\of  $, be globally Lipschitz. This is 
obtained by modifying, if needed, this function outside a compact set, as the saturation function
does in (\ref{LP3}).

%%%%%%%%%%%%%%%%%%%%%%%%%%%%%%%%%%%%%%%%%%%%%%%%%%%%%%%%%%%%%%%%%%%%%%%%%%%%%%%%%%%%%%%%%%%%%%%%%%%%%%%%%%%%%%%%%%%%%%%%%%%%%%%%%%%%%%%%%%%%%%%%%%%
\section{Modifying $\ohi$ and $\setphitau$ given by Assumption $\hypo$}~
\label{disc_P3}
\\
The sufficient conditions, given in Sections \ref{jac_comp} and \ref{diffeo_ext}, 
to solve  Problem \ref{defP1} and Problem \ref{defP2}
 in order to fulfill  
the requirements 
of Proposition \ref{prop_CSobsWithoutInversion},
impose
conditions on the dimensions or on the domain of injectivity $\Ouvt$
which are not always satisfied : contractibility for Jacobian complementation
\jacobcomplet% a Jacobian completion
and diffeomorphism extension, limited number of pairs $(m,n)$ for the 
$\tilde P[m,n]$ problem, etc. Expressed in terms of our initial problem, these conditions are 
limitations on the data $f$, $h$ and $\ohi$ that we considered. In the following, we 
show by means of examples that, in some 
cases, these data can be modified in such a way that our various tools apply 
and give a satisfactory solution. Such modifications are possible 
since we restrict our attention
to  system solutions which remain in $\OuvEx$. Therefore 
we can modify arbitrarily the data $f$, $h$ and $\ohi$ 
outside this set. For example we can add arbitrary ``fictitious'' components to 
the measured output $y$ as long as their value is known on $\OuvEx$.

\subsection{For contractibility}

It may happen that the set $\Ouvt$ attached to $\ohi$ is not contractible,
for example due to an observability singularity. We have seen that
Jacobian complementation
\jacobcomplet% a Jacobian completion
and image extension may be 
prevented by this 
(see Theorem \ref{theo_Waz} and Remark \ref{rem_contractible_diffeoExt}).
A possible approach to overcome this difficulty when we know the system trajectories stay away from
the singularities is to add a fictitious output traducing this information :

\begin{example}[Continuation of Example \ref{ex_osci_2}]\normalfont
\label{ex_osci_5}
The observer we have obtained at the end of Example \ref{ex_osci_2} 
for the harmonic oscillator with unknown frequency is not 
satisfactory in particular because of the singularity at $\hat 
x_1=\hat x_2=0$. To overcome this difficulty we add, to the given 
measurement $y=x_1$, the following 
\\[0.5em]\null \hfill $\displaystyle 
y_2=h_2(x) \;=\;  \wp (x_1,x_2)\,   {{x_3}}
$\hfill \null \\
with
\\\null \hfill $\displaystyle 
\wp (x_1,x_2)\;=\; \max\left\{ 0 ,  \frac{1}{r^2} - (x_1^2 + x_2^2) 
\right\}^4 \ .
$\hfill \null \\[0.5em]
By construction this function is zero on $\OuvEx$ and $y_2$ can thus be considered as an extra measurement.
%meaning that we 
%do not change the data on this set. 
The interest of $y_2$ is to give access to $x_3$ even at the singularity $x_1=x_2=0$. 
Indeed, consider the new function $\ohi$ defined as
\begin{equation}
\ohi(x)\;=\; \left(
x_1 \,,\: x_2 \,,\: -x_1x_3 \,,\: -x_2x_3 \,,\:  \wp (x_1,x_2)\,  {{x_3}}
\right) \ .
\label{ohi_HG6}
\end{equation}
$\ohi$ is $C^1$ on $\RR^3$ and its Jacobian is~:
\begin{equation}
\label{LP12}
\frac{\partial \ohi}{\partial x}(x) \; = \;
\left(\begin{array}{ccc}
    1    &    0     &   0
\\
    0    &    1     &   0
\\
   -x_3   &    0     &  -x_1
\\
    0    &   -x_3    &  -x_2
\\
\frac{\partial \wp }{\partial x_1}  {{x_3}}
&
\frac{\partial \wp }{\partial x_2}  {{x_3}}
& \wp 
\end{array}\right) \ ,
\end{equation}
which has full rank $3$ on $\RR^3$, since $\wp (x_1,x_2) \neq 0$ when $x_1=x_2=0$.
It follows that the singularity has disappeared and this new $\ohi$ is an injective immersion
on the entire $\RR^3$ which is contractible.

We have shown in Example \ref{ex-Waz} how Wazewski's algorithm 
allows us to get in this case a $C^2$ function $\gamma :\RR^3 \to \RR^4$ 
satisfying~:
$$
\det\left(\frac{\partial \ohi}{\partial x}(x) \  \gamma 
(x)\right)\;\neq\; 0
\qquad \forall x\in \RR^3
\  .
$$
This gives us 
$
\ohia(x,w)\;=\; \ohi(x) + \gamma (x) w
$
which is a $C^2$-diffeomorphism on $\RR^3\times\BR_\varepsilon(0)$,
with $\varepsilon $ sufficiently small.

Furthermore, $\Ouvxw=\RR^3\times\BR_\varepsilon(0)$ being now diffeomorphic to $\RR^5$, Corollary 
\ref{cor_theo_diff_ext} applies and provides an extension $\ohie$ of $\ohia$ satisfying  Problems \ref{defP1} and \ref{defP2}.
\hfill$\triangle$\end{example}

\subsection{For a solvable $\tilde P[m,n]$ problem}
If we are in a case that cannot be reduced to a solvable $\tilde P[m,n]$ problem, we may try to 
modify $m$ by adding  arbitrary rows to $\frac{\partial \ohi}{\partial x}$. We illustrate this technique 
with the following example.

\begin{example}[Continuation of Example \ref{ex_osci_5}]\normalfont
\label{ex_osci_3} 
In Example \ref{ex_osci_5}, by adding the fictitious measured output 
$y_2=h_2(x)$, we have obtained another function $\ohi$ 
for the harmonic oscillator with unknown frequency which is an 
injective immersion on $\RR^3$. In this case, we have $n=3$ and $m=5$ 
which gives a pair not in  (\ref{LP11}). But, as already 
exploited in Example \ref{ex_osci_2}, the first $2$ 
rows of the Jacobian $\frac{\partial \ohi}{\partial x}$ in 
(\ref{LP12}) are independent for all $x$ in $\RR^3$. 
It follows that 
our Jacobian complementation
\jacobcomplet% a Jacobian completion
problem reduces to complement the vector
$
%\label{6}
\left(
  -x_1
,
 -x_2
,
 \wp (x_1,x_2)
\right)
$.
This is a problem  with pair $(3,1)$ which is not in  (\ref{LP11}) either.
Instead, the pair $(4,1)$ is,
meaning that the following 
vector can be complemented
 via a universal formula
$
%\label{10}
\left(
  -x_1
,
 -x_2
,
 \wp (x_1,x_2)
,
0
\right)\  .
$
We have added a zero component, without changing the full rank 
property. Actually this vector is extracted from the Jacobian of
\\[0.7em]\null \hfill $\displaystyle 
\ohi(x)\;=\; \left(
x_1 \,,\: x_2 \,,\: -x_1x_3 \,,\: -x_2x_3 \,,\:  \wp (x_1,x_2)\,  {{x_3}} \,,\: 0
\right) 
\  .
$\refstepcounter{equation}\label{LP14}\hfill$(\theequation)$\\[0.7em]
In the high gain 
observer paradigm, this zero we add can come from another (fictitious) measured output
$
y_3=0
\  .
$
A complement  of
$
\left(
  -x_1
,
 -x_2
,
 \wp (x_1,x_2)
,
0
\right)
$
is
\\[0.7em]\null \hfill $\displaystyle 
\left(\begin{array}{@{}ccc@{\;}c@{\;}c@{\;}c@{}}
 x_2 & -\wp  & 0
\\
 -x_1 & 0 & -\wp 
\\
 0 & -x_1 & -x_2
\\
 \wp  & x_2 & -x_1
\end{array}\right)
$\hfill \null \\[0.3em]
It gives the 
function
\\[0.7em]$\displaystyle
\ohia(x,w)=\left(
\vrule height 0.6em depth 0.6em width 0pt
\right.
x_1\: , \:  x_2\: , \: [-x_1x_3 + x_2 w_1 - \wp (x_1,x_2) w_2 ]\: , \: [-x_2x_3-x_1 w_1 - \wp (x_1,x_2) w_3 ]\: , \: 
$\hfill \null
\\[0.3em]\null\hfill $\displaystyle
\left.[\wp (x_1,x_2) {{x_3}}-x_1w_2 -x_2 w_3 ]\: , \: [\wp (x_1,x_2)w_1 + x_2w_2 -x_1w_3) ]
\vrule height 0.6em depth 0.6em width 0pt
\right) .
% \,[-x_1x_3 + x_2 w_1 - \wp (x_1,x_2) w_2 ]\, ,
$\\[0.7em]
% the Jacobian of which is~:
% $$
% \frac{\partial \ohie}{\partial x}(x,w)\;=
% \left(\begin{array}{@{}ccc@{\;}c@{\;}c@{\;}c@{}}
%     1    &    0     &   0 & 0 & 0 & 0
% \\
%     0    &    1     &   0 & 0 & 0 & 0
% \\
%    -x_3 -\frac{\partial \wp }{\partial x_1} w_2   &    w_1-\frac{\partial \wp }{\partial x_2} w_2    &  -x_1 & x_2 & -\wp  & 0
% \\
%     -w_1 -\frac{\partial \wp }{\partial x_1}w_3    &   -x_3 - \frac{\partial \wp }{\partial x_2} w_3 &  -x_2 & -x_1 & 0 & -\wp 
% \\
% \frac{\partial \wp }{\partial x_1} {{x_3}} -w_2 &
%  \frac{\partial \wp }{\partial x_2} {{x_3}} -w_3& \wp  & 0 & -x_1 & -x_2
% \\
%  \frac{\partial \wp }{\partial x_1}w_1  -w_3  &    \frac{\partial \wp }{\partial x_2}w_1 +w_2    &   0 & \wp  & x_2 & -x_1
% \end{array}\right)
% $$
% with determinant
the Jacobian determinant of which is
$
(x_1^2 + x_2^2 + \wp (x_1,x_2)^2)^2
$ 
which is nowhere $0$ on $\RR^6$.
Hence $\ohia$ is locally invertible. Actually it is
diffeomorphism from $\RR^6$ onto $\RR^6$  since we 
can express $\ox=\ohia(x,w)$ as
\\[0.7em]\null \hfill $\displaystyle 
\left(\begin{array}{@{}c@{}}
x_1 \\ x_2
\end{array}\right)=
\left(\begin{array}{@{}c@{}}
\ox_1 \\ \ox_2
\end{array}\right)
\  ,\quad 
\left(\begin{array}{@{}cccc@{}}
-\ox_1 & \ox_2 & -\wp (\ox_1,\ox_2) & 0
\\
-\ox_2 & -\ox_1 & 0 & -\wp (\ox_1,\ox_2)
\\
\wp  (\ox_1,\ox_2)& 0 & -\ox_1 & -\ox_2
\\
0 & \wp (\ox_1,\ox_2) & \ox_2 & -\ox_1
\end{array} \right)\left(\begin{array}{@{}c@{}}
x_3 \\ w_1 \\ w_2 \\ w_3
\end{array}\right)
=
\left(\begin{array}{@{}c@{}}
\ox_3 \\ \ox_4 \\ \ox_5 \\ \ox_6
\end{array}\right)\  ,
$\hfill \null \\[0.7em]
% %
% \begin{eqnarray*}
% &\displaystyle 
% x_1=\ox_1\; ,\  
% x_2=\ox_2\; ,\  
% \\[1em]
% &\displaystyle 
% \left(\begin{array}{cccc}
% -\ox_1 & \ox_2 & -\wp (\ox_1,\ox_2) & 0
% \\
% -\ox_2 & -\ox_1 & 0 & -\wp (\ox_1,\ox_2)
% \\
% \wp  (\ox_1,\ox_2)& 0 & -\ox_1 & -\ox_2
% \\
% 0 & \wp (\ox_1,\ox_2) & \ox_2 & -\ox_1
% \end{array} \right)\left(\begin{array}{c}
% x_3 \\ w_1 \\ w_2 \\ w_3
% \end{array}\right)\;=\; 
% \left(\begin{array}{c}
% \ox_3 \\ \ox_4 \\ \ox_5 \\ \ox_6
% \end{array}\right)\  .
% \end{eqnarray*}
where the matrix on the left is invertible by construction. Since $\ohia(\RR^6)=\RR^6$, there is no need of an image extension and
we simply take $\ohie=\ohia$.
To have all the assumptions of Proposition \ref{prop_CSobsWithoutInversion} satisfied, it remains to find a function $\of $ such that
$(\ohede{x},\of  )$ is in the set $\setphitau$, the function $\ohede{x}$
being the $x$-component of the inverse of $\ohie$. Exploiting the fact that,
for $x$ in $\OuvEx$, we have
\\[0.7em]\null \hfill $\displaystyle 
\dot y_2=\dot{\overparen{\wp (x_1,x_2) {{x_3}}}}\;=\; 0
\quad ,\qquad 
\dot y_3\;=\; 0 \ ,
$\hfill \null \\[0.7em] 
the high gain observer paradigm gives the function
 $$
 \of(\ox,\hat x,y)\;=\; 
 \left(\begin{array}{@{}c@{}}
 \ox_2 + \ell k_1 (y-\hat x_1)
 \\
 \ox_3 + \ell^2 k_2 (y-\hat x_1)
 \\
 \ox_4 + \ell^3 k_3 (y-\hat x_1)
 \\
 \sat(\hat x_1\hat x_3^2) + \ell^4 k_4 (y-\hat x_1)
 \\
 - a \, \ox_5
 \\
 -b \, \ox_6
 \end{array}\right) 
 $$
where the function $\sat$ is defined in (\ref{LP13}) and $a$ and 
$b$ are arbitrary strictly positive real numbers.
With picking $\ell$ large enough, it can be paired with any function $\oh:\RR^6\to \RR^6$ which is 
locally Lipschitz, and thus in particular with $\ohede{x}$. 
Therefore, Proposition \ref{prop_CSobsWithoutInversion} applies and gives
the 
following observer for the harmonic oscillator with unknown frequency
\\[1em]$\displaystyle 
\left(\begin{array}{@{}c@{}}
\dot{\hat x}_1
\\
\dot{\hat x}_2
\\
\dot{\hat x}_3
\\
\dot{\hat w}_1
\\
\dot{\hat w}_2
\\
\dot{\hat w}_3\end{array}\right)
= 
%%%%%%%
\left(\begin{array}{@{}ccc@{\;}c@{\;}c@{\;}c@{}}
    1    &    0     &   0 & 0 & 0 & 0
\\
    0    &    1     &   0 & 0 & 0 & 0
\\
   -\hat x_3 -\frac{\partial \wp }{\partial \hat x_1} \hat w_2   &    \hat w_1-\frac{\partial \wp }{\partial  x_2} \hat w_2    &  -\hat x_1 & \hat x_2 & -\wp  & 0
\\
    -\hat w_1 -\frac{\partial \wp }{\partial \hat x_1}\hat w_3    &   -\hat x_3 - \frac{\partial \wp }{\partial  x_2} \hat w_3 &  -\hat x_2 & -\hat x_1 & 0 & -\wp 
\\
\frac{\partial \wp }{\partial x_1} {{\hat x_3}} -\hat w_2 &
\frac{\partial \wp }{\partial x_2} {{\hat x_3}}-\hat w_3& \wp  & 0 & -\hat x_1 & -\hat x_2
\\
 \frac{\partial \wp }{\partial x_1}\hat w_1  -\hat w_3  &    \frac{\partial \wp }{\partial x_2}\hat w_1 +\hat w_2    &   0 & \wp  & \hat x_2 & -\hat x_1
\end{array}\right)^{-1}
%%%%%%%
\times
$\hfill \null \refstepcounter{equation}\label{LP19}(\theequation)\\\null \hfill $\displaystyle 
\times
\left(\begin{array}{@{}c@{}}
\hat x_2 + \ell k_1 (y-\hat x_1)
\\{}
[-\hat x_1\hat x_3 + \hat x_2 \hat w_1 - \wp (\hat x_1,\hat x_2) \hat w_2 ] + \ell^2 k_2 (y-\hat x_1)
\\{}
[-\hat x_2\hat x_3-\hat x_1 \hat w_1 - \wp (\hat x_1,\hat x_2) \hat w_3 ] + \ell^3 k_3 (y-\hat x_1)
\\
\sat(\hat x_1\hat x_3^2) + \ell^4 k_4 (y-\hat x_1)
\\
- a \, [\wp (\hat x_1,\hat x_2) {{\hat x_3}}-\hat x_1\hat w_2 -\hat x_2 \hat w_3 ]
\\
-b \, [\wp (\hat x_1,\hat x_2)\hat w_1 + \hat x_2\hat w_2 -\hat x_1\hat w_3) ]
\end{array}\right) .
$\\[1em]
It is globally defined and globally convergent for any solution 
of the oscillator initialized in the set $\OuvEx$ given in (\ref{eq_OuvsExpl}).
% $$
% \OuvEx\;=\; \left\{x\in\RR^3\,  :\: x_1^2+x_2^2 \in 
% \left]\textstyle\frac{1}{r},r\right[\: ,\;
% x_3\in ]0,r[\right\}
% \  .
% $$

Observer (\ref{LP19}) is an illustration of what can be 
obtained by using in a very nominal way our tools. We do not claim 
any property for it. For example, by using another design, an observer of dimension $2$, globally 
convergent on $\OuvEx$, can be obtained.
\hfill$\triangle$\end{example}

In this example we have made the Jacobian
complementation possible by increasing $m$ with augmenting the number of coordinates of $\ohi$.
Actually if we augment $\ohi$ with $n$ zeros the possibility of a Jacobian
complementation is guaranteed. Indeed pick any $C^1$  function $B$  the values of which are $m\times m$ matrices with positive 
definite symmetric part, we can 
complement $\left(\begin{array}{@{}c@{}}
\frac{\partial \ohi}{\partial x}
\\
0
\end{array}\right)$ which is full column rank with
$ 
\gamma \;=\; \left(\begin{array}{@{}c@{}}
-B 
\\
\frac{\partial \ohi}{\partial x}^\top
\end{array}\right)$.
This follows from
 the  identity (Schur complement) involving invertible matrices
$$
\left(\begin{array}{cc}
\frac{\partial \ohi}{\partial x} & -B
\\
0 & \frac{\partial \ohi}{\partial x}^\top
\end{array}\right)
\left(\begin{array}{cc}
0 & I
\\
I & B^{-1} \frac{\partial \ohi}{\partial x}
\end{array}\right)
\;=\; 
\left(\begin{array}{cc}
-B & 0
\\
\frac{\partial \ohi}{\partial x}^\top & \frac{\partial \ohi}{\partial x}^\top B^{-1} \frac{\partial \ohi}{\partial x}
\end{array}\right)
\  .
$$
So we have here a universal method to solve our Problem \ref{defP1}. Its drawback is that the dimension of 
the state increases by $m$, instead of $m-n$.

%%%%%%%%%%%%%%%%%%%%%%%%%%%%%%%%%%%%%%%%%%%%%%%%%%%%%%%%%%%%%%%%%%%%%%%%%%%%%%%%%%%%%%%%%%%%%%%%%%%%%%%%%%%%%%%%%%%%%%%%%%%%%%%%%%%%%%%%%%%%%%%%%%%

\section{Conclusion}
We have presented a method to express the dynamics of an observer in
 preferred  coordinates
  enlarging its domain of validity and possibly
avoiding the difficult left-inversion of
an injective immersion.
It assumes the knowledge of an injective immersion
and a converging observer for the immersed system. 

 The idea is not 
to modify this observer dynamics but to map it back to the 
 preferred  coordinates in a different way. Our construction involves  two 
tools :
the augmentation of an injective immersion into a diffeomorphism through
a Jacobian complementation
\jacobcomplet% a Jacobian completion
and the extension of the image of the obtained
diffeomorphism to enlarge the domain
where the observer solutions can go without encountering singularities.

For the Jacobian complementation
 we rely on results by Wazewski 
\cite{Waz} and Eckmann \cite{Eck}. They allows us to build a 
diffeomorphism by augmenting the  preferred  coordinates with new ones
and to write the given observer dynamics in these augmented
coordinates.

For the
diffeomorphism extension, we have proposed our own method 
inspired from diffeotopies \cite[Chapter 8]{Hir}
and $h$-cobordism \cite[pages 2,  7 to 14 and 16 to 18]{Milnor}.

We have assumed the system is time-invariant and 
autonomous. Adding time-variations is not a problem but dealing with 
exogenous inputs is more complex. This is in part due to the fact 
that, as far as we know, the theory of observers, in 
presence of such inputs, relying on immersion into a space of 
larger dimension, as high gain observers or nonlinear Luenberger 
observers, is not
satisfactory enough
 yet. Progress on this topic has to be made 
before trying to extend our results.

\startmodif
One very important question which remains to be addressed is about optimizing the 
observer performance. In our framework it consists in an appropriate selection of the given ``raw'' observer, i.e. the 
functions $\varphi$ and $\tau ^*$ in (1.2), and the diffeomorphism $\tau _e$ for optimizing a cost 
expressing the quality of the estimated quantities %to be estimated 
with respect to what they are made for. For such a 
task, remaining in an ideal context with  no modelling error and no measurement disturbance, allows only to 
address the transient behavior of the state estimate. To be interesting for practice, at least as important if not more important is the 
long range dependence of the state estimates on unmodelled effects.
\stopmodif

%%%%%%%%%%%%%%%%%%%%%%%%%%%%%%%%%%%%%%%%%%%%%%%%%%%%%%%%%%%%%%%%%%%%%%%%%%
%%%%%%%%%%%%%%%%%%%%%   APPENDIX  %%%%%%%%%%%%%%%%%%%%%%%%%%%%%%%%%%%%%%

\appendix

\section{Proof of Proposition \ref{prop_CSobsWithoutInversion}}
\label{app_CSobsWithoutInversion}
Let $(x_0,(\hat x_0,\hat w_0))$ be arbitrary in $\OuvEx\times\Ouvxw$ but 
such that $X(x_0,t)$ solution of (\ref{syst}) is defined and remains in $\OuvEx$ for $t$ in $[0,+\infty )$. 
Let $[0,T[$ be the right maximal interval of definition of the 
solution
$(X(x_0,t),\hat X(\hat x_0,\hat w_0,t;y_{x_0}),\hat W(\hat x_0,\hat w_0,t;y_{x_0}))$
when considered with values in  $\OuvEx\times\Ouvxw$.
Assume for the time being $T$ is
finite. Then, when $t$ goes to $T$, either $(\hat X(\hat x_0,\hat w_0,t;y_{x_0}),\hat W(\hat x_0,\hat w_0,t;y_{x_0}))$ 
goes to infinity or to the boundary of $\Ouvxw$.
\sloppy
By construction
$t\mapsto  \oX(t):= \ohie\left(\hat X(\hat x_0,\hat w_0,t;y_{x_0}),\hat W(\hat x_0,\hat w_0,t;y_{x_0})\right)$ is a solution of (\ref{eq_sysObs}) on $[0,T[$ with $\tau=\ohede{x}$. From 
assumption $\hypo$ and since $(\of,\ohede{x})$ is in $\setphitau$, it can be extended as a solution defined on 
$[0,+\infty [$ when considered with values in 
$\RR^m=\ohie(\Ouvxw)$. This implies that $\oX(T)$ is well 
defined in $\RR^m$.
Since, with (\ref{n9}), the inverse $\ohe$ of $\ohie$ is a diffeomorphism defined on 
$\RR^m$, we obtain
$
\lim_{t\to T}
\left(\hat X(\hat x_0,\hat w_0,t;y_{x_0}),\hat W(\hat x_0,\hat w_0,t;y_{x_0})\right)\;=\; 
\ohe(\oX(T))
$,
which is an interior point of $\ohe(\RR^m)=\Ouvxw$.
This point being neither a boundary point nor at infinity, we have a 
contradiction. It follows that $T$ is infinite.

Finally, with assumption $\hypo$, we have~:
$$
\lim_{t\to +\infty }
\left|\ohie\left(\hat X(\hat x_0,\hat w_0,t;y_{x_0}),\hat W((\hat x_0,\hat w_0,t;y_{x_0})\right)-\ohi(X(x_0,t))\right|\;=\; 0
\  .
$$
Since $X(x_0,t)$ remains in $\OuvEx$, $\ohi(X(x_0,t))$ equals $\ohie\left(X(x_0,t),0\right)$ and remains
in the compact set $\ohi(\CL(\OuvEx))$. So there exists a compact subset
$\mathbf{C}$ of $\RR^m$ and a time $t_\mathbf{C}$ such that $\ohie\left(\hat X(\hat x_0,\hat w_0,t;y_{x_0}),\hat W(\hat x_0,\hat w_0,t;y_{x_0})\right)$ is in $\mathbf{C}$ for all $t>t_\mathbf{C}$.  Since $\ohie$ is a diffeomorphism, its 
inverse $\ohe$ is Lipschitz on the compact set $\mathbf{C}$. This implies (\ref{1}).

%%%%%%%%%%%%%%%%%%%%%%%%%%%%%%%%%%%%%%%%%%%%%%%%%%%%%%%%%%%%%%%%%%%%%%%%%
%%%%%%%%%%%%%%%%%%%%%%%%%%%%%%%%%%%%%%%%%%%%%%%%%%%%%%%%%%%%%%%%%%%%%%%%

\section{Proof of Lemma \ref{lem_CS_P1}}
\label{app_CS_P1}
The fact that $\ohia$ is an immersion for $\varepsilon$ small enough is established in \cite{AndEytPra}.
We now prove it is injective.
Let $\varepsilon _0$ be a strictly positive real number such that the Jacobian 
of $\ohia(x,w)$ in (\ref{ohie}) is invertible for any $(x,w)$ in 
$\CL(\Ouvs\times\BR_{\varepsilon_0}(0))$. Since 
$\CL(\Ouvs\times\BR_{\varepsilon_0}(0))$ is compact, not to contradict the Implicit 
function Theorem, there exists a strictly positive real number 
$\delta $ such that any two pairs $(x_a,w_a)$ and $(x_b,w_b)$ in 
$\CL(\Ouvs\times\BR_{\varepsilon_0}(0))$ which satisfy
\begin{equation}
\label{LP5}
\ohia(x_a,w_a)\;=\; \ohia(x_b,w_b) \quad , \quad (x_a,w_a)\;\neq \;(x_b,w_b)
\end{equation}
satisfies also
$
|x_a-x_b|\;+\; |w_a-w_b|\; \geq \; \delta .$
On another hand, since $\ohi$ is continuous and injective on $\CL(\Ouvs)\subset 
\Ouvt$,
it has an inverse which is uniformly continuous on the compact set
 $\ohi(\CL(\Ouvs))$ (see \cite[\S 16.9]{Bartle}).
It follows that there exists a strictly positive real number $\eta $ such that
$$
|x_a-x_b|< \frac{\delta }{2}
\qquad \forall \,  (\ohi(x_a), \ohi(x_b)) \in \ohi(\CL(\Ouvs))^2:\,  
|\ohi(x_a)-\ohi(x_b)| < \eta 
\  .
$$
But if (\ref{LP5}) holds with $w_a$ and $w_b$ in 
$\BR_{\varepsilon}(0)$ with $\varepsilon \leq \varepsilon _0$, we 
have
$$
\delta - 2\varepsilon  \:  \leq \:  |x_a-x_b|
\   ,\quad \   
|\ohi(x_a)-\ohi(x_b)|
\: =\:  
|\gamma (x_a)w_a-\gamma (x_b)w_b|
\: \leq \:  2\varepsilon \!\sup_{x\in\CL(\Ouvs)}|\gamma (x)|
\  .
$$
We have a contradiction for all
$\varepsilon \leq \min\left\{\frac{3\delta }{4},\frac{\eta }{2\varepsilon 
\sup_{x\in\CL(\Ouvs)}|\gamma (x)|}\right\}$. So (\ref{LP5}) cannot hold for such $\varepsilon $'s,
i.e. $\ohia$ is injective on $\Ouvs\times\BR_\varepsilon(0)$.
% On another hand, since $\ohi$ is continuous and injective on $\CL(\Ouvs)\subset 
% \Ouvt$, by following the same arguments as those for establishing (\ref{LP4}), we can prove 
% the existence of a class $\mathcal{K}$ function $\beta$ such 
% that we have
% $$
% |x_a-x_b|\; \leq \; \beta(|\ohi(x_a)-\ohi(x_b)|)
% \qquad \forall (x_a,x_b)\in\CL(\Ouvs)^2
% \  .
% $$
% It follows that, if (\ref{LP5}) holds with $w_a$ and $w_b$ in 
% $\BR_{\varepsilon}(0)$ with $\varepsilon \leq \varepsilon _0$, we 
% have
% $$
% \delta \;-\; 2\varepsilon \; \leq \; |x_a-x_b|\; \leq \; 
% \beta(|\ohi(x_a)-\ohi(x_b)|)
% \;=\; 
% \beta(|\gamma (x_a)w_a-\gamma (x_b)w_b|)
% \; \leq \; \beta\left(2\varepsilon \sup_{x\in\CL(\Ouvs)}|\gamma 
% (x)|\right)
% $$
% But there exists 
% a strictly positive real number $\varepsilon  \leq
% \varepsilon _0$ 
% such that we 
% have
% $$
% \delta \;-\; 2\varepsilon \; >\; \beta\left(2\varepsilon \sup_{x\in\CL(\Ouvs)}|\gamma 
% (x)|\right)
% \  .
% $$
% For all such $\varepsilon$, (\ref{LP5}) 
% is impossible. This proves that $\ohia$ is injective on $\Ouvs\times\BR_\varepsilon(0)$.

%%%%%%%%%%%%%%%%%%%%%%%%%%%%%%%%%%%%%%%%%%%%%%%%%%%%%%%%%%%%%%%%%%%%%%%%%%%%%%%%%%%%%%
%%%%%%%%%%%%%%%%%%%%%%%%%%%%%%%%%%%%%%%%%%%%%%%%%%%%%%%%%%%%%%%%%%%%%%%%%%%%%%%%%%%%
\startlongue{
\section{Proof of ``only if'' in Theorem \ref{theo_EckmannTilde}}
\label{app_EckmannTilde}
The following theorem is due to Eckmann.
\begin{theorem}[\cite{Eck}]
For $m>n$, there exists a continuous function $\vartau \in\RR^{m \, \times \, n} \mapsto \tilde 
\gamma_1(\vartau)\in \RR^m$ with non zero values 
and satisfying
\\[0.6em]\null \hfill $\displaystyle 
\tilde \gamma_1(\vartau)^T \vartau=0
\qquad \forall \vartau\in\RR^{m \, \times \, n}:\,  
\textsf{Rank}(\vartau)=n
$\hfill \null \\[0.6em]
if and only if  $(m,n)$ is in 
one of the following $4$ pairs
\\[0.7em]
\null \hfill $(\geq 2,m-1)$
\hfill or\hfill 
$(\textrm{even},1)$
\hfill or\hfill 
$(7,2)$
\hfill or\hfill 
$(8,3)$
\refstepcounter{equation}\label{LP11b}\hfill$(\theequation)$\\[0.7em]
\end{theorem}
With Remark \ref{rem_jacCompOrth}, any pair $(m,n)$ for 
which $\tilde P[m,n]$ is solvable must be one in the list (\ref{LP11b}).
The pair $(\geq 2,m-1)$ is in the list (\ref{LP11}). For the pair $(\textrm{even},1)$,
we need to find $m-1$ vectors to complement the given one into an invertible matrix.
After normalizing the vector $\vartau$ so that it belongs
to the unit sphere $\mathbb{S}^{m-1}$ and projecting each vector
$\gamma_i(\vartau)$ of $\gamma (\vartau)$ onto the orthogonal complement of $\vartau$, this complementation 
problem is
equivalent to asking whether $\mathbb{S}^{m-1}$ is parallelizable
(since the $\gamma_i(\vartau)$ will be a basis for the tangent space at $\vartau$
for each $\vartau\in \mathbb{S}^{m-1}$).  It turns out that this problems
admits solutions only for $m=4$ or $m=8$ (see \cite{BotMil}). So in the pairs $(\textrm{even},1)$ only 
$(4,1)$ and $(8,1)$ are in the list (\ref{LP11}).

Finally,  since $\tilde{P}[6,1]$ has no solution, the pairs
$(7,2)$ and $(8,3)$ cannot be in the list (\ref{LP11}).
Indeed
let $\vartau$ be a full column rank $(m-1)\times(n-1)$ matrix.
$\left(\begin{array}{@{}c@{\hskip 0.5em}c@{}}
\vartau & 0
\\
0 & 1
\end{array}\right)$ is a full column rank $m\times n$ matrix. If
if $\tilde{P}[m,n]$ has a solution, there exist a continuous $(m-1)\times (m-n)$ matrix function $\tilde 
\gamma $ and a continuous row 
vector functions $a^T$ such that
% $\left(\begin{array}{@{}l@{}}
% \tilde \gamma (\vartau)\\a(\vartau)^T
% \end{array}\right)$
% with $a^T$ a row vector, is
% an $m\times (m-n)$ matrix
such that
$\left(\begin{array}{@{}l@{\hskip 0.1em}c@{\hskip 0.7em}c@{}}
\tilde \gamma (\vartau) &\vartau & 0 \\ a(\vartau)^\top & 0 & 1
\end{array}\right)$ is invertible. This implies that
$\left(\begin{array}{@{}c@{\hskip 0.5em}c@{}}
\tilde \gamma (\vartau) &\vartau 
\end{array}\right)$ is also invertible. So if $\tilde{P}[m,n]$ has a solution, $\tilde{P}[m-1,n-1]$ must have one.
}\stoplongue
%%%%%%%%%%%%%%%%%%%%%%%%%%%%%%%%%%%%%%%%%%%%%%%%%%%%%%%%%%%%%%%%%%%%%%%%%%%%%%
%%%%%%%%%%%%%%%%%%%%%%%%%%%%%%%%%%%%%%%%%%%%%%%%%%%%%%%%%%%%%%%%%%%%%%%%%%%%%%%%%
\startlongue{
\section{End of proof of Theorem \ref{theo_Waz}}
\label{sec4}
We want to  show that a continuous function $\gamma$ 
making $P$ in (\ref{eqWaz}) invertible can be modified into a smoother one giving the 
same invertibility property. 
Let $\gamma_i$ denote the $i$th 
column of $\gamma$.
We start with modifying $\gamma _1$ into 
$\tilde{\gamma }_1$.
Since $\vartau$, $\gamma$ and the determinant are continuous, for any $x$ in $\Ouvt$,
there exists a strictly positive real number $r_x$, such that, may be after changing $\gamma _1$ into $-\gamma _1$,
\begin{equation}
\det\left(\vartau(y) \ \gamma_1 (x) \ \gamma_{2:m-n}(y) \right) > 0 \ 
, \qquad \forall y \in \BR_{r_x}(x) \ ,
\label{eq_detNeigh}
\end{equation}
where  $\gamma_{i:j}$ denotes the matrix composed of the $i^{th}$ to $j^{th}$ columns of $\gamma$.
The family of sets $\left(\BR_{r_x}(x)\right)_{x\in \Ouvt}$ is an open cover of $\Ouvt$.
Therefore, by \cite[Theorem 2.1]{Hir}, there exists  a subordinate 
$C^\infty$ partition of unity, i.e.~there exist a family of $C^\infty$ functions $\psi_x : \Ouvt\rightarrow \RR_{\geq 0}$
such that
\begin{eqnarray}
\Supp \left(\psi_x\right) \subset \BR_{r_x}(x) \quad \forall x \in \Ouvt \ , \label{property1}& \\
\{\Supp \left(\psi_x\right) \}_{x\in\Ouvt} \textrm{ is locally finite} \ , \label{property2} & \\
\sum_{x\in\Ouvt}{\psi_x(y)} = 1  \quad \forall y \in \Ouvt \label{property3}& \ .
\end{eqnarray}
With this, we define the function $\tilde{\gamma}_1$ on $\Ouvt$ by
$$
\tilde{\gamma}_1 (y) = \sum_{x\in\Ouvt} \psi_x(y) \gamma_1(x) \ .
$$
This function is well-defined and $C^\infty$ on $\Ouvt$ because the sum is finite at
each point according to (\ref{property2}). Using multi-linearity of 
the determinant, we have, for all $y$ in $\Ouvt$,
$$
\det\left(\vartau(y) \ \tilde{\gamma}_1 (y) \ \gamma_{2:m-n}(y) \right)
=\sum_{x\in\Ouvt} \psi_x(y)\det\left(\vartau(y) \ \gamma_1 (x) \ \gamma_{2:m-n}(y) \right) \ .
$$
Thanks to (\ref{property2}), at each point $y$ in $\Ouvt$, there is a finite number of
$\psi_x(y)$ which are not zero. 
Also, the right hand side is the sum of non negative terms because of 
(\ref{eq_detNeigh}) and the non negativeness of the $\psi _x$, and one 
of these terms is strictly positive because of (\ref{eq_detNeigh}) 
and (\ref{property3}).
Therefore, we can replace the continuous function $\gamma_1$ by the $C^\infty 
$ function $\tilde{\gamma}_1$ as a first 
column of $\gamma$. Then we follow exactly the same procedure for 
$\gamma _2$ with this modified $\gamma $. By proceeding this way, one column after the other,
we get our result.
}\stoplongue

%%%%%%%%%%%%%%%%%%%%%%%%%%%%%%%%%%%%%%%%%%%%%%%%%%%%%%%%%%%%%%%%%%%%%%%%%%%%%%
%%%%%%%%%%%%%%%%%%%%%%%%%%%%%%%%%%%%%%%%%%%%%%%%%%%%%%%%%%%%%%%%%%%%%%%%%%%%%%%%%

\section{Construction of a diffeomorphism from an open set to $\RR^m$}
\label{ann_lem1}
We use the following notations:
\\
The complementary, closure and boundary
of a set $S$ are denoted $S^c$, $\CL(S)$ and 
$\partial S$, respectively. 
The Hausdorff distance $d_H$ between two sets $A$ and $B$ is defined by~:
\\[0.4em]\null \hfill $\displaystyle 
d_H(A ,B)\; = \; 
\max\left\{
\sup_{z_A \in A }\inf_{z_B \in B}|z_A -z_B|
\,  ,\,  
\sup_{z\in A}\inf_{z_B \in B }|z_A -z_B|
\right\} 
\  .
$\hfill \null \\[0.6em]
$Z(z,t)$ denotes the (unique) solution, at time $t$, to
$
\dot z = \chi(z)
$
going trough $z$ at time $0$
and
$\  
\Seps \;=\; \underset{t \in [0,\varepsilon]}{\bigcup} Z(\partial E,t) 
$.

\begin{lemma}
	\label{lem_annE}
	Let $E$ be an open strict subset 
	of $\RR^m$ verifying $\condC$, with a $C^s$ vector field $\chi$ and a $C^s$ mapping $\kappa$.
	There exists a strictly positive (maybe infinite) real number $\varepsilon _{\infty}$ such that,
	for any $\varepsilon $ in $[0,\varepsilon _{\infty}[$,
	there exists a $C^s$-diffeomorphism $\phi$: $\RR^m\rightarrow E$, such 
	that
	$$
	\phi(z)=z \quad \forall z\in \Eeps = E \cap \Sepscomp \quad , \qquad d_H(\partial \Eeps ,\partial E) \; \leq \;\varepsilon \,   \sup_z |\chi(z)| \ .
	$$
\end{lemma}

\startjournal{%
	\indent{\it Proof.}\ignorespaces
	We give here only a sketch. A complete proof is given in \cite{BerPraAndSIAMHal}.
	
	\noindent The definition of $\Eeps$ using the flow generated by $\chi$ gives
	$d_H(\partial \Eeps ,\partial E)\; \leq \; \varepsilon \,  \sup_\zeta |\chi(\zeta)|$.
	Let $
	\tde (z)
	$ in $[-\deps ,+\infty [$ be the time needed by a solution to $
	\dot z = \chi(z)
	$ starting from $z$ in $\left(\Edeps\right)\comp$ (which is open) to reach
	$\partial E$, i.e. satisfying:
	\\[0.6em]\null \hfill $\displaystyle 
	\kappa \left(Z(z,
	\tde (z)
	)\right)=0
	\quad \Longleftrightarrow\quad 
	Z(z,
	\tde (z)
	)\in \partial E.
	$\hfill \null \\[0.6em]
	This gives a $C^s$ function on $\left(\Edeps\right)\comp$
	which satisfies
	\\[0.6em]\null \hfill $\displaystyle 
	\tde(z)\in [-\deps  ,-\varepsilon ]\qquad \forall z\in \Eeps\setminus\Edeps
	\  .
	$ \hfill \null \\[0.6em]
	We extend by continuity to $\RR^m$ by letting
	$\tde (z)=-\deps$ for all $z$ in $\Edeps$.
	Let $\temps:\RR\to \RR$ be a function
	such that the function $t\mapsto \temps(t)-t$ is a $C^s$ (decreasing) diffeomorphism from $\RR$ onto 
	$]0,+\infty [$ mapping $[-\varepsilon ,+\infty [$ onto $]0,\varepsilon ]$ and being ``minus'' identity
	on $]-\infty ,-\varepsilon]$ (i.e. $\temps$ is zero on $]-\infty ,-\varepsilon]$). 
	%and $[\varepsilon ,+\infty [$.
	Then $\phi$ defined below is a $C^s$ diffeomorphism which satisfies the required properties
	\\[0.7em]\null \hfill $\displaystyle
	\phi(z)=
	\left\{
	\begin{array}{ll}
		Z\left(z,
		\temps(\tde (z))
		\right), &\quad \textrm{if} \ z \in \left(\Eeps \right) ^c
		\  ,
		\\[0.5em]
		z, &\quad \textrm{if} \ z \in \Eeps
		\  .
	\end{array}
	\right.
	$\hfill \null
}\stopjournal

\startlongue{%
	\begin{proof}
		According to Condition $\condC$, $\chi$ is bounded and $K_0$ is a compact
		subset  of the open set $E$. It follows that there exists a strictly positive (maybe infinite) real number $\varepsilon _{\infty}$
		such that
		$$
		Z(z,t)\not\in K_0
		\qquad \forall (z,t)\in \partial E\times [0,\deps _{\infty}[
		\  .
		$$
		In the following $\varepsilon $ is a real number in  $[0,\varepsilon _{\infty}[$.
		
		We introduce the notations
		$$
		\Sdeps= \underset{t \in [0,\deps ]}{\bigcup} Z(\partial E,t) 
		\quad ,\qquad
		\Edeps = E \cap \Sdepscomp
		$$
		and establish some properties.

		\noindent
		--
		\textit{$E$ is forward invariant for $\chi$}. This is a 
		direct consequence of points $\condC$.\ref{point1} and $\condC$.\ref{point2}.

		\noindent-- \textit{$\Sdeps$ is closed.}  Take a sequence $(z_k)$ of
		points in
		$\Sdeps$
		converging to $z^*$.  By definition of $\Sdeps$, 
		there exists a sequence $(t_k)$,
		such that~:
		$$
		t_k \in [0,\deps ]
		\qquad \textrm{and}\qquad 
		Z(z_k,-t_k)\in \partial E
		\qquad \forall k \in \NN
		\  .
		$$ 
		Since $[0,\deps ]$ is compact, one
		can extract a
		subsequence $(t_{\sigma(k)})$ converging to $t^*$
		in
		$[0,\deps ]$, and by continuity
		of the function $(z,t)\mapsto Z(z,-t)$,
		$(Z(z_{\sigma(k)},t_{\sigma(k)}))$ tends to $Z(z^*,- t^*) $
		which is in
		$\partial E$, since $\partial E$ is closed.  Finally,
		because $t^*$ is in $[0,\deps ]$, $z^*$ is in $\Sdeps$ by 
		definition.

		\noindent-- \textit{$\Sdeps $ is contained in $\CL(E)$}. Since, $E$ is forward invariant for $\chi$, 
		and so is $\CL(E)$
		(see \cite[Theorem 16.3]{Hahn}).
		This implies
		$$
		\partial E\; \subset\; \Sdeps = \underset{t \in 
			[0,\deps ]}{\bigcup} Z(\partial E,t) \; \subset\; \CL(E)
		\;=\; E\cup \partial E \ .
		$$
		At this point, it is useful to note that, because $\Sdeps $ is a closed subset of $\CL(E)$ and $E$ is open, 
		we have $\Sdeps \cap E \; = \; \Sdeps \! \setminus\! \partial E$.
		This implies~:
		\begin{equation}
		E \! \setminus\! \Edeps = \left(\Edeps\right)\comp \cap E = (E\comp \cup \Sdeps) \cap E = \Sdeps \cap E = \Sdeps \! \setminus\! \partial E,
		\label{EminusEeps}
		\end{equation}
		and
		$
		E= \Edeps
		%\underset{\neq}{\cup} 
		\renewcommand{\arraystretch}{0.2}
		\raise -0.2em\hbox{$\begin{array}{@{\; }c@{\; }}
				\cup\\\scriptscriptstyle\neq 
			\end{array}$}(\Sdeps \! \setminus\! \partial E)
		$.
		%\end{itemize}
		
		With all these properties at hand, we define now two functions $\tde$ and 
		$
		\theta 
		$.
		The assumptions of global attractiveness of the closed set $K_0$ 
		contained in $E$ open, of 
		transversality 
		of $\chi$ to
		$\partial E$, and the property of forward-invariance of $E$, imply
		that, for all $z$ in $E\comp$, there exists a unique non negative real number $
		\tde (z)
		$
		satisfying:
		$$
		\kappa \left(Z(z,
		\tde (z)
		)\right)=0
		\quad \Longleftrightarrow\quad 
		Z(z,
		\tde (z)
		)\in \partial E.
		$$ 
		The same arguments in reverse time allow us to see that,
		for all $z$ in $\Sdeps$, $
		\tde (z)
		$ exists,
		is unique and in $[-\deps ,0]$. This way, the function $z\to 
		\tde (z)
		$ is defined 
		on $\left(\Edeps\right)\comp$.
		Next,
		for all $z$
		in
		$\left(\Edeps\right)\comp$, we define :
		$$
		\theta (z)=Z(z,\tde (z)).
		$$
		Thanks to the 
		transversality assumption, the Implicit Function Theorem implies the
		functions $z\mapsto 
		\tde (z)
		$ and $z\mapsto 
		\theta (z)
		$ are 
		$C^s$  on $\left(\Edeps\right)\comp$.

		\begin{remark}
			\label{rem2}
			\normalfont
			$\kappa$ having constant rank $1$ in a 
			neighborhood of $\partial E$, this set is a closed, regular 
			submanifold of $\RR^m$. The arguments above show that $z\mapsto (
			\theta (z)
			,
			\tde (z)
			)$ is a 
			diffeomorphism between $E\comp$ and
			$\partial E\times [0,+\infty[$. Since $\partial E$ is a 
			deformation retract of $E\comp$ and the open 
			unit ball is diffeomorphic to $\RR^m$ \cite{GonTos}, if $E$ were bounded, 
			$E\comp$ could be seen as a $h$-cobordism
			between $\partial E$ and the unit sphere $\SS^{m-1}$ and
			$
			\tde 
			$ as a Morse function with no critical point in $E\comp$. See \cite{Milnor} for 
			instance.
		\end{remark}
		
		Now we evaluate $
		\tde (z)
		$ for $z$ in $\partial \Sdeps $. Let $z$ be
		arbitrary in $\partial \Sdeps$ and therefore in $\Sdeps $ which is
		closed.  Assume its corresponding $
		\tde (z)
		$ is in
		$]-\depsÊ, 0[$. The Implicit Function Theorem shows
		that $z\mapsto 
		\tde (z)
		$ and $z\mapsto
		\theta (z)
		$ are defined and continuous
		on a neighborhood of $z$.  Therefore, there exists
		a strictly 
		positive real number $r$
		satisfying
		$$
		\forall y \in \BR_r(z)\: ,\; \exists t_y \in ]-\deps ,0[\: :\;  Z(y,t_y)
		\in \partial E
		\  .
		$$
		This implies that the neighborhood $\BR_r(z) $ of $z$ is contained in 
		$\Sdeps$, in contradiction with
		the fact that $z$ is on the boundary
		of $\Sdeps$.
		This shows that, for all $z$ in $\partial \Sdeps$, $
		\tde (z)
		$ is
		either $0$ or $-\deps $.
		We write this as
		$$
		\left(\partial \Sdeps \right)_i=\left\{ z \in 
		\Sdeps 
		\,  :\: 
		\tde (z)
		=-\deps 
		\right\}
		\quad ,\qquad 
		\partial \Sdeps = \partial E \cup \left(\partial \Sdeps \right)_i
		\  .
		$$
		
		Now we want to prove $\partial \Edeps \subset\left(\partial \Sdeps 
		\right)_i$. To obtain this result, we start by showing~:
		\begin{equation}
		\label{LP18}
		\partial \Edeps \cap \partial E =\emptyset \quad \textrm{and} \quad \partial \Edeps \subset \partial \Sdeps
		\ .
		\end{equation}
		Suppose 
		the existence of $z$ in $\partial \Edeps \cap \partial E$.
		$z $ being in $\partial E$, its corresponding $
		\tde (z)
		$ is $0$. By 
		the Implicit Function Theorem, there exists
		a strictly 
		positive real number $r$ such that,
		$$
		\forall y \in \BR_r(z)\: ,\; \exists t_y \in \left]\textstyle-
		\varepsilon ,\varepsilon 
		\right[\: :
		\displaystyle Z(y,t_y) \in \partial E 
		\  .
		$$
		But, by definition, any $y$, for which there exists $t_y$ in 
		$]-\varepsilon ,0]$, is in $\Sdeps $. If instead $t_y$ is 
		strictly positive, then necessarily $y$ is in $E\comp$, because $E$ is forward-invariant for $\chi$
		and a solution starting in $E$ cannot reach $\partial E$ in positive finite time. We have obtained~:
		$
		\BR_r(z)\subset \Sdeps \cup E\comp = (\Edeps )\comp
		$.
		$\BR_r(z)$ being a neighborhood of $z$, this contradicts the fact 
		that $z$ is in the boundary of $\Edeps$.  
		
		At this point, we have proved that $\partial \Edeps \cap \partial E =\emptyset$, and,
		because $\Edeps$ is contained in $E$, this implies
		$\partial \Edeps \subset E$. With this, (\ref{LP18}) will be established by proving that we have
		$
		\partial \Edeps \subset \partial \Sdeps
		$.
		Let $z$ be arbitrary in $\partial \Edeps$ and therefore in $E$ 
		which is open. There exists a strictly 
		positive real number $r$ such that we have~:
		\\[0.7em]\null \hfill $\displaystyle 
		\emptyset\; \neq \; \BR_r(z)\cap \Edeps \; =\; \BR_r(z)\; \cap\; \left(E\cap \Sdepscomp\right)
		\ ,\quad
		\emptyset\; \neq \; \BR_r(z)\cap \left(\Edeps\right)\comp \;=\; \BR_r(z)\; \cap\; \left(E\comp \cup \Sdeps \right)
		\ ,\quad \BR_r(z)\subset E\ .
		$\hfill \null \\[0.7em]
		This implies $
		\BR_r(z)\; \cap\; \Sdepscomp\; \neq \; \emptyset
		$ and $ 
		\BR_r(z)\; \cap\; \Sdeps \; \neq \; \emptyset
		$
		and therefore that $z$ is in $\partial \Sdeps $.
		
		We have established
		$
		\partial \Edeps \cap \partial E\;=\; \emptyset
		$, $\partial \Edeps \subset \partial \Sdeps$ 
		and 
		$\partial \Sdeps \;=\; \partial E \cup \left(\partial \Sdeps \right)_i$.
		This does imply~:
		\begin{equation}
		\partial \Edeps \; \subset\; \left(\partial \Sdeps \right)_i
		\;=\; \{z\in E\,  :\: 
		\tde (z)
		=-\deps  \}
		\  .
		\label{boundaryEeps}
		\end{equation}
		This allows us to extend by continuity the definition of $\tde $ to $\RR^m$ by letting
		$$
		\tde (z)\;=\; -\deps \qquad \forall  z\in \Edeps
		\  .
		$$
		All the properties we have established for $\Sdeps$ and $\Edeps$ hold also for $\Seps$ and $\Eeps$. In 
		particular, we have
		\begin{equation}
		\label{newLP2}
		\tde(z)\in [-\deps  ,-\varepsilon ]\qquad \forall z\in \Eeps\setminus\Edeps
		\  .
		\end{equation}

		Thanks to all these preparatory steps, we are finally ready to define 
		a function $\phi:\,  \RR^m\to E$.
		Let $\temps:\RR\to \RR$ be a function
		such that the function $t\mapsto \temps(t)-t$ is a $C^s$ (decreasing) diffeomorphism from $\RR$ onto 
		$]0,+\infty [$ mapping $[-\varepsilon ,+\infty [$ onto $]0,\varepsilon ]$ and being ``minus'' identity
		on $]-\infty ,-\varepsilon ]$, i.e.
		\\[0.6em]\null \hfill $\displaystyle 
		\temps(t)-t\;=\; -t\quad \forall t\leq -\varepsilon
		\  .
		$\hfill \null \\[0.6em]
		We have
		\\[0.7em]\null \hfill $\displaystyle 
		\temps(t) > t\qquad \forall t\in \RR
		\quad ,\qquad 
		\temps(\tde(z))\;=\; 0\qquad \forall z\in \Eeps\setminus\Edeps
		\  .
		$\hfill \null \refstepcounter{equation}\label{7}$(\theequation)$\\[0.1em]
		We let~:
		\\[0.1em]\null \hfill $\displaystyle 
		\phi(z)=
		\left\{
		\begin{array}{ll}
		Z\left(z,
		\temps(\tde (z))
		\right), &\quad \textrm{if} \ z \in \left(\Edeps \right) ^c
		\  ,
		\\[0.5em]
		z, &\quad \textrm{if} \ z \in \Edeps
		\  .
		\end{array}
		\right.
		$\hfill \null \\[0.7em]
		The image of $\phi$ is contained in $E$ since we have (\ref{7}), $\Edeps\; \subset\; E$ and~:
		$$
		Z(z,\tde (z))\; \in\; \partial E
		\quad ,\qquad 
		Z(z,t)\; \in\; E\qquad \forall (z,t)\in \partial E\times \RR_{>0}
		\  .
		$$
		Like the functions $Z$, $\temps$, and $\tde$, the function $\phi$ is $C^s$ on the interior of 
		$\left(\Edeps\right)\comp$.
		Also, since (\ref{7}) implies
		\begin{equation}
		\label{newLP1}
		\phi(z)\;=\; z\qquad  \forall z\in \Eeps\setminus\Edeps
		\  ,
		\end{equation}
		$\phi$ is trivially $C^s$ on $\Eeps$ and therefore on $\left(\Edeps\right)\comp\cup \Eeps=\RR^m$.
		
		We now show that $\phi$ is invertible.
		Because of (\ref{newLP1}), this is trivial on $\Eeps$. 
		Let $y$ be arbitrary in $E\cap \left(\Edeps\right)\comp=E\cap \Sdeps $.
		To $y$ corresponds $\tde(y)$ in  the interval $[-\deps ,0[$.
		Thus, $-\tde(y)$ is in $]0,\deps ]$, image of
		$[-\deps  ,+\infty [$ by the $C^s$ diffeomorphism $t\mapsto \temps(t)-t$.
		Hence there exists $\sde(y)$ in
		$[-\deps ,+\infty [$ satisfying
		\begin{equation}
		\label{newLP3}
		\temps(\sde(y))\;-\; \sde(y)\;=\; -\tde(y)
		\  .
		\end{equation}
		Moreover, (\ref{newLP2}) implies that for $y$ in $\Eeps\setminus\Edeps$ subset of $E\cap 
		\left(\Edeps\right)\comp$, we have
		$$
		\sde(y)\;=\; \tde(y)
		$$
		So with letting
		$$
		\sde(y)\;=\; \tde (y)\;=\; -\deps \qquad \forall  y\in \Edeps
		$$
		we have defined a function $\sde: E \to [-\deps, + \infty [$, which thanks to the implicit function 
		theorem, is $C^s$ and satisfies (\ref{newLP3}).
		
		This allows us to define properly $\phi^{-1}:\RR^m\to E$ as~:
		$$
		\phi^{-1}(y)=
		Z\left(y,-\temps(\sde(y))\right) \ .
		$$
		By composition, this function is $C^s$ and it is an inverse of $\phi$ in particular because, with 
		(\ref{newLP3}), we have
		$$
		\tde(Z(y,-\nu(\sde(y))))=\tde(Z(y,\tde(y)-\sde(y)))=\sde(y)
		\qquad \forall y\in E
		\  .
		$$
		This gives
		\\[0.7em]$\displaystyle 
		\phi(Z(y,-\nu(\sde(y)))=Z(Z(y,-\nu(\sde(y))),\nu(\tde(Z(y,-\nu(\sde(y))))))
		$\hfill \null \\\null \hfill $\displaystyle =
		Z(Z(y,-\nu(\sde(y))),\nu(\sde(y)))=y
		$\\[0.7em]
		
		All this implies
		$\phi$ is a $C^s$-diffeomorphism from $\RR^m$ to $E$.
		
		Finally, we note that, for any point $z_\varepsilon $ in 
		$\partial E_\varepsilon$, there exists a point $z$ in $\partial E$ 
		satisfying~:
		\\[0.3em]\null \hfill $\displaystyle 
		|z_\varepsilon -z|\;=\; \left|\int_0^{\varepsilon } \chi(Z(z ,s)) ds\right|
		\; \leq \; \varepsilon \,  \sup_\zeta |\chi(\zeta)| 
		\  .
		$\hfill \null \\[0.7em]
		And conversely, for any $z$ in $\partial E$, there exist 
		$z_\varepsilon $ in $\partial \Eeps$
		satisfying~:
		$$
		|z_\varepsilon -z|\;=\; \left|\int_0^{\varepsilon } \chi(Z(z ,s)) ds\right|
		\; \leq \; \varepsilon \,  \sup_\zeta |\chi(\zeta)| 
		\  .
		$$
		It follows that, with $\varepsilon$ as small as needed,
		\\[0.7em]\null \hfill $\displaystyle 
		d_H(\partial E_\varepsilon ,\partial E)\; \leq \; \varepsilon \,  \sup_\zeta |\chi(\zeta)| 
		$\hfill \null \refstepcounter{equation}\label{dH}$(\theequation)$\  
		\end{proof}
}\stoplongue

Lemma \ref{lem1} is a direct consequence of Lemma \ref{lem_annE} if we 
pick $\varepsilon _\infty$, maybe infinite, satisfying
$$
Z(z,t)\not\in K
\qquad \forall (z,t)\in \partial E\times [0,\deps _{\infty}[
\  .
$$
$\varepsilon _\infty$ can be chosen strictly positive since $d(K,\partial E)$ is non 
zero and $\chi$ is bounded.

\section{Proof of case b) of  Theorem \ref{theo_diff_ext}}
\label{app_diff_ext_case_b}
To complete the proof of Theorem \ref{theo_diff_ext}, we use another technical result.
\begin{lemma}[Diffeomorphism extension from a ball]
	\label{lem2}
	Consider a $C^2$ diffeomorphism
	$\lambdaa:\BR_R(0)\rightarrow\lambdaa(\BR_R(0))\subset\RR^m$, with $R$ a strictly positive real number.  For any real number
	$\varepsilon$ in $]0,1[$, there exists a diffeomorphism
	$
	\lambdae
	:\RR^m\rightarrow\RR^m$ satisfying
	$$
	\lambdae
	(\varchi) = \lambdaa(\varchi)\qquad \forall \varchi \in 
	\CL(\BR_{\frac{R}{1+\varepsilon}}(0))\  .
	$$
\end{lemma}

\begin{proof}
It sufficient to prove that \cite[Theorem 8.1.4]{Hir} applies. 
We let
$$
U=\BR_{\frac{R}{1+\frac{\varepsilon}{2}}}(0)
\quad ,\qquad 
A=\CL(\BR_{\frac{R}{1+\varepsilon}}(0))
\quad ,\qquad 
I=\left]-\frac{\varepsilon}{2},1+\frac{\varepsilon}{2}\right[
\  ,
$$
and, without loss of generality we may assume that $\lambdaa(0) = 0$.

Then, consider the function $F:U\times I\rightarrow\RR^m$ defined as
$$
F(z,t) = \left(\frac{\partial \lambdaa}{\partial z}(0)\right)^{-1}\frac{\lambdaa(zt)}{t} \ , \ \forall t \in I 
\setminus \{0\} \ , \quad F(z,0)=z \ .
$$
We start by showing that $F$ is an isotopy of $U$. 
\begin{list}{}{%
\parskip 0pt plus 0pt minus 0pt%
\topsep 0pt plus 0pt minus 0pt%\topsep 0.5ex plus 0pt minus 0pt%
\parsep 0pt plus 0pt minus 0pt%
\partopsep 0pt plus 0pt minus 0pt%
\itemsep 0pt plus 0pt minus 0pt%\itemsep 0.5ex plus 0pt minus 0pt
\settowidth{\labelwidth}{$\quad  \bullet$}%
\setlength{\labelsep}{0.5em}%
\setlength{\leftmargin}{\labelwidth}%
\addtolength{\leftmargin}{\labelsep}%
}
	\item[$\bullet$] For any $t$ in $I$, the function 
	$z\mapsto F_t = F(z,t)$ is an embedding from 
	$U$ onto $F_t(U)\subset\RR^m$.
	Indeed, for any pair $(z_a,z_b)$ in $U^2$ satisfying
	$
	F(z_a,t)=F(z_b,t)
	$,
	we obtain
	$
	\lambdaa(z_at)=\lambdaa(z_bt)
	$
	where $(z_at,z_bt)$  is in $U^2$. The function $\lambdaa$ being 
	injective on this set, we have $z_a=z_b$ which establishes 
	$F_t$ is injective.
	Moreover, we have:
	$$
	\frac{\partial F_t}{\partial z}(z)
	=  \left(\frac{\partial \lambdaa}{\partial z}(0)\right)^{-1}\frac{\partial \lambdaa}{\partial z}(zt)
	\quad \forall t\in I \setminus \{0\} 
	\quad ,\qquad 
	\frac{\partial F_0}{\partial z}(z) =  \Id.
	$$
	Hence, $F_t$ is full rank on $U$ and therefore an 
	embedding.
	
	\item[$\bullet$] For all $z$ in $U$, the function $t\mapsto F(z,t)$ is $C^1$.
	This follows directly from the fact that, the function $\lambdaa$ being $C^2$, and $\lambdaa(0)=0$, we have
	$$
	\frac{\lambdaa(zt)}{t} = \frac{\partial \lambdaa}{\partial z}(0)z + z'\left(\frac{\partial^2 \lambdaa}{\partial z\partial z}(0)\right) z \frac{t}{2} + \circ (t) \ .
	$$
	In particular, we obtain 
	$
	\frac{\partial F}{\partial t}(z,t) = 
	\left(\frac{\partial \lambdaa}{\partial z}(0)\right)^{-1} \rho(z,t)
	$ 
	where
	$$
	\rho(z,t) = \frac{1}{t^2}\left[\frac{\partial \lambdaa}{\partial z}(zt)zt -\lambdaa(zt)\right]
	\quad \forall t \in I \setminus \{0\}
	\ ,\quad 
	\rho(z,0) = \frac{1}{2} z'\left(\frac{\partial^2 \lambdaa}{\partial z\partial z}(0)\right) z \ .
	$$
\end{list}

Moreover, for all $t$ in $I$, the function $z\mapsto \frac{\partial F}{\partial t}(z,t)$ is locally Lipschitz and therefore gives rise to an ordinary differential equation with unique solutions. 

Also the set 
$
\bigcup_{(z,t)\in U\times I}  \{ (F(z,t),t) \}
$
is open. This follows from Brouwer's Invariance theorem since the function $(z,t)\mapsto (F(z,t),t)$
is a diffeomorphism on the open set $U\times I$.
With \cite[Theorem 8.1.4]{Hir}, we know there exists a diffeotopy $F_e$
from $\RR^m\times I$ onto $\RR^m$ which satisfies $F_e=F$ on $A\times[0,1]$.
Thus, the diffeomorphism $\lambdae= F_e(., 1)$ defined on $\RR^m$ onto $\RR^m$ verifies
$\lambdae(z)=F_e(z,1)=F(z,1)=\lambdaa(z)$  for all $z\in A$.
\hfill\end{proof}

We now place ourselves in the case b) of Theorem \ref{theo_diff_ext}, namely we suppose that $\ohia$ is $C^2$ and $\OuvDef$ is $C^2$-diffeomorphic to $\RR^m$. 
	Let $\phi_1: \OuvDef \rightarrow \RR^m$ denote the corresponding diffeomorphism.
	Let $R_1$ be a strictly positive real number such that the open ball $\BR_{R_1}(0)$ contains $\phi_1(K)$.
	Let $R_2$ be a real number strictly larger than $R_1$.
	With Lemma \ref{lem1} again, and since $\BR_{R_2}(0)$ verifies condition $\condC$, there exists of
	$C^2$-diffeomorphism  $\phi_2:\RR^m \to \BR_{R_2}(0)$ satisfying
	$
	\phi_2(z)\;=\; z
	$
	for all $z$ in $\BR_{R_1}(0)$.
	At this point, we have obtained a $C^2$-diffeomorphism 
	$
	\phi=\phi_2\circ \phi_1:\OuvDef \rightarrow\BR_{R_2}(0).
	$
	Consider $\lambda=\ohia\circ\phi^{-1} :
	\BR_{R_2} (0)\rightarrow \ohia(\OuvDef)
	\quad  \left( = \lambda (\BR_{R_2} (0))\right)$.
	According to Lemma \ref{lem2}, we can extend $\lambda$ to
	$\lambda_e : \RR^m \rightarrow \RR^m$ such that $\lambda_e=\ohia\circ\phi^{-1}$ 
	on $\BR_{R_1}(0)$. Finally, consider $
	\ohie
	=\lambda_e \circ \phi_1 : \OuvDef 
	\rightarrow \RR^m$. Since, by construction of $\phi_2$, $\phi=\phi_1$ on $\phi _1^{-1}(\BR_{R_1}(0))$
	which contains $K$, we have $
	\ohie
	=\ohia$ on $K$.

%%%%%%%%%%%%%%%%%%%%%%%%%%%%%%%%%%%%%%%%%%%%%%%%%%%%%%%%%%%%%%%%%%%%%%%%%%%%%%%%%%%%%%%%%%%%%%%%%%%
%%%%%%%%%%%%%%%%%%%%%%%%%%%%%%%%%%%%%%%%%%%%%%%%%%%%%%%%%%%%%%%%%%%%%%%%%%%%%%%%%%%%%%%%%%%%%%%%%%

\section{Proof of Lemma \ref{lemCond}}
\label{ann_lem3}
The compact $K_0$ being globally asymptotically attractive and interior
to $E$ which is forward invariant, $E$ is globally attractive.  It is
also stable due to the continuity of solutions with respect to initial
conditions uniformly on compact time subsets of the domain of
definition.  So it is globally asymptotically stable.  It follows from
\cite[Theorem 3.2]{Wil} that there exist $C^\infty$ functions
$V_K:\RR^m\to \RR_{\geq 0}$ and $V_{E}:\RR^m\to \RR_{\geq 0}$ which
are proper on $\RR^m$ and a class $\mathcal{K}_\infty $ function
$\alpha $ satisfying
    \begin{eqnarray*}
        &
\displaystyle
\alpha (d(\varchi ,K_0)) \leq  V_K(\varchi ) \  ,\quad 
\alpha (d(\varchi ,E)) \leq  V_{E}(\varchi ) 
\qquad \forall \, \varchi  \in \RR^m\;,
\\
&\displaystyle 
 V_K(\varchi )  =  0 \quad \forall z\in K_0 \quad ,\quad    V_{E}(\varchi ) = 0\quad \forall \, \varchi  \in E\;,\\
&\displaystyle
\frac{\partial V_K}{\partial \varchi }(\varchi )\, \chi(\varchi )  \leq   -V_K(\varchi ) 
\  ,\quad 
\frac{\partial V_{E}}{\partial \varchi }(\varchi )\,\chi(\varchi )
        \leq   -V_{E}(\varchi ) \qquad \forall \, \varchi  \in \RR^m\;.
    \end{eqnarray*}
With $\overline{d}$ an arbitrary strictly positive real number,
the notations
$$
v_{E} \;=\;    \sup_{\varchi \in\RR^m: \,  d(\varchi ,E)\leq \overline{d}  }V_K(\varchi  )
\quad ,\qquad 
\mu =\frac{\alpha (\overline{d}   )}{2v_{E}}
\  ,
$$
and since $\alpha $ is of class $\mathcal{K}_\infty $, we obtain 
the implications
\\[0.7em]$\displaystyle 
V_{E}(\varchi ) \!+ \!\mu   V_K(\varchi ) \!=\!\alpha (\overline{d}   )
\quad \Rightarrow \quad 
\alpha (d(\varchi ,E))\!\leq\! V_{E}(\varchi ) \!\leq \!\alpha (\overline{d}   )
$\hfill \null \\[0.3em]\null \hfill $\displaystyle  \Rightarrow \quad 
d(\varchi ,E)\leq \overline{d}   
\quad \Rightarrow \quad 
V_K(\varchi  )\leq v_{E}
\:   .
$\\[0.7em]
With our definition of $\mu $, this yields also
$$
\alpha (\overline{d}   )-\mu \,   V_K(\varchi )
= V_{E}(\varchi )  
\quad \Rightarrow \quad 
0\; <\; \frac{\alpha (\overline{d}   )}{2} \leq  V_{E}(\varchi )
\quad \Rightarrow \quad 
0 < d(\varchi ,E)\leq \overline{d} 
\;  .  
$$
On the other hand, with the compact notation
$
\Lyap (\varchi ) = V_{E}(\varchi )+\mu V_K(\varchi )
$,
we have
$
\frac{\partial \Lyap }{\partial 
\varchi }(\varchi ) \, \chi (\varchi ) 
\leq 
-\Lyap (\varchi )
$, for all $\varchi  \in \RR^m$.
All this implies that the sublevel set
$
\E=\{\varchi  \in \RR^m:\,  \Lyap (\varchi ) <  \alpha ( \overline{d}    )\}
$
is contained in $\{\varchi \in\RR^m\!:\,  d(\varchi ,E) \in [0, \overline{d}]\}$
and that $\CL(E)$ is contained in $\E$. Besides, $\E$ verifies condition $\condC$
with the vector field $\chi$ and the function
$\kappa =\Lyap-\alpha (\overline{d})$.

%%%%%%%%%%%%%%%%%%%%%%%%%%%%%%%%%%%%%%%%%%%%%%%%%%%%%%%%%%%%%%%%%%%%%%%%%
%%%%%%%%%%%%%%%%%%%%%%%%%%%%%%% BIBLIO %%%%%%%%%%%%%%%%%%%%%%%%%%%%%%%%%
\bibliography{biblio}

\begin{thebibliography}{10}

\bibitem{AndEytPra}
V.~Andrieu, J.-B. Eytard, and L.~Praly.
\newblock Dynamic extension without inversion for observers.
\newblock {\em IEEE Conference on Decision and Control}, pages 878--883, 2014.

\bibitem{AndPra}
V.~Andrieu and L.~Praly.
\newblock On the existence of a {K}azantzis--{K}ravaris / {L}uenberger
  observer.
\newblock {\em SIAM Journal on Control and Optimization}, 45(2):432--456, 2006.

\bibitem{AstPra}
D.~Astolfi and L.~Praly.
\newblock Output feedback stabilization for siso nonlinear systems with an
  observer in the original coordinate.
\newblock {\em IEEE Conference on Decision and Control}, pages 5927 -- 5932,
  2013.

\bibitem{Bartle}
R.~Bartle.
\newblock {\em The elements of real analysis}.
\newblock John Wiley \& Sons, 1964.

\bibitem{BerPraAnd}
P.~Bernard, L.~Praly, and V.~Andrieu.
\newblock Tools for observers based on coordinate augmentation.
\newblock {\em IEEE Conference on Decision and Control}, 2015.

\bibitem{BesTic}
G.~Besan\c{c}on and A.~Ticlea.
\newblock An immersion-based observer design for rank-observable nonlinear
  systems.
\newblock {\em IEEE Transactions on Automatic Control}, 52(1):83--88, 2007.

\bibitem{BotMil}
R.~Bott and J.~Milnor.
\newblock On the parallelizability of the spheres.
\newblock {\em Bullettin of American Mathematical Society}, 64(3):87--89, 1958.

\bibitem{DezaBusGauRak}
F.~Deza, E.~Busvelle, J.P. Gauthier, and D.~Rakotopara.
\newblock High gain estimation for nonlinear systems.
\newblock {\em Systems \& Control Letters}, 18:295--299, 1992.

\bibitem{Dug}
J.~Dugundgi.
\newblock {\em Topology}.
\newblock Allyn and Bacon, 1966.

\bibitem{Eck}
B.~Eckmann.
\newblock {\em Mathematical survey lectures $1943-2004$}.
\newblock Springer, 2006.

\bibitem{GauHamOth}
J.-P. Gauthier, H.~Hammouri, and S.~Othman.
\newblock A simple observer for nonlinear systems application to bioreactors.
\newblock {\em IEEE Transactions on Automatic Control}, 37(6):875--880, 1992.

\bibitem{GauKup}
J-P. Gauthier and I.~Kupka.
\newblock {\em Deterministic observation theory and applications}.
\newblock Cambridge University Press, 2001.

\bibitem{GonTos}
S.~Gonnord and N.~Tosel.
\newblock {\em Calcul Diff\'{e}rentiel}.
\newblock Ellipses, 1998.

\bibitem{Hahn}
W.~Hahn.
\newblock {\em Stability of Motion}.
\newblock Springer-Verlag, 1967.

\bibitem{Hir}
M.~Hirsch.
\newblock {\em Differential topology}.
\newblock Springer, 1976.

\bibitem{KazKra}
N.~Kazantzis and C.~Kravaris.
\newblock Nonlinear observer design using {L}yapunov's auxiliary theorem.
\newblock {\em Systems and Control Letters}, 34:241--247, 1998.

\bibitem{KhaPra}
H.~K. Khalil and L.~Praly.
\newblock High-gain observers in nonlinear feedback control.
\newblock {\em Int. J. Robust. Nonlinear Control}, 24, April 2013.

\bibitem{KhaSab}
H.K. Khalil and A.~Saberi.
\newblock Adaptive stabilization of a class of nonlinear systems using
  high-gain feedback.
\newblock {\em IEEE Transactions on Automatic Control}, 32(11):1031--1035,
  1987.

\bibitem{LafBusGau}
F.~Lafont, E.~Busvelle, and J.-P. Gauthier.
\newblock An adaptive high-gain observer for wastewater treatment systems.
\newblock {\em Journal of Process Control}, 21(6):893--900, 2011.

\bibitem{LevMar}
J.~Levine and R.~Marino.
\newblock Nonlinear system immersion, observers and finite-dimensional filters.
\newblock {\em Systems \& Control Letters}, 7:133--142, 1986.

\bibitem{MagPas}
M.~Maggiore and K.M. Passino.
\newblock A separation principle for a class of non uniformly completely
  observable systems.
\newblock {\em IEEE Transactions on Automatic Control}, 48, July 2003.

\bibitem{Milnor}
J.~Milnor.
\newblock {\em Lectures on the $h$-cobordism Theorem. Notes by L. Siebenmann
  and J. Sondow}.
\newblock Princeton University Press, 1965.

\bibitem{RapMal}
A.~Rapaport and A.~Maloum.
\newblock Design of exponential observers for nonlinear systems by embedding.
\newblock {\em International Journal of Robust and Nonlinear Control},
  14:273--288, 2004.

\bibitem{Shos}
A.~Shoshitaishvili.
\newblock On control branching systems with degenerate linearization.
\newblock {\em IFAC Symposium on Nonlinear Control Systems}, pages 495--500,
  1992.

\bibitem{Tor}
A.~Tornambe.
\newblock Use of asymptotic observers having high gains in the state and
  parameter estimation.
\newblock {\em IEEE Conference on Decision and Control}, 2:1791--1794, 1989.

\bibitem{Waz}
T.~Wazewski.
\newblock {\em Sur les matrices dont les \'{e}l\'{e}ments sont des fonctions
  continues}, volume~2.
\newblock Composito Mathematica, 1935.
\newblock p. 63-68.

\bibitem{Wil}
F.~W. Wilson.
\newblock Smoothing derivatives of functions and applications.
\newblock {\em Transactions American Mathematical Society}, 139:413--428, 1969.

\end{thebibliography}
%\bibliography{../../biblio}
\bibliographystyle{plain}

\end{document}